\documentclass{article}
\usepackage{amssymb}
\usepackage{amsmath}
\usepackage{harvard}

\newtheorem{theorem}{Theorem}

\newtheorem{corollary}[theorem]{Corollary}

\newtheorem{definition}[theorem]{Definition}
\newtheorem{example}[theorem]{Example}

\newtheorem{proposition}[theorem]{Proposition}
\newtheorem{remark}[theorem]{Remark}

\newenvironment{proof}[1][Proof]{\noindent\textbf{#1.} }{\ \rule{0.5em}{0.5em}}
\numberwithin{theorem}{section}
\numberwithin{equation}{section}
\input{tcilatex}

\begin{document}

\title{Distinguished Torsion, Curvature and Deflection Tensors in the
Multi-Time Hamilton Geometry}
\author{Gheorghe Atanasiu and Mircea Neagu}
\date{}
\maketitle

\begin{abstract}
The aim of this paper is to present the main geometrical objects on the dual
1-jet bundle $J^{1\ast }(\mathcal{T},M)$ (this is the polymomentum phase
space of the De Donder-Weyl covariant Hamiltonian formulation of field
theory) that characterize our approach of multi-time Hamilton geometry. In
this direction, we firstly introduce the geometrical concept of a nonlinear
connection $N$ on the dual 1-jet space $J^{1\ast }(\mathcal{T},M)$. Then,
starting with a given $N$-linear connection $D$ on $J^{1\ast }(\mathcal{T}%
,M) $, we describe the adapted components of the torsion, curvature and
deflection distinguished tensors attached to the $N$-linear connection $D.$
\end{abstract}



\section{Introduction}

\hspace{4mm} It is well known that the 1-jet spaces are the basic
mathematical objects used in the study of classical and quantum field
theories. For this reason, the differential geometry of $1$-jet bundles was
intensively studied by a lot of authors like, for example, (in chronological
order) Saunders [22], Asanov [4], Neagu and Udri\c{s}te [21], [19].

In the last decades, numerous physicists and geometers were preoccupied with
the development of a branch of mathematical-physics, which is situated at
the junction of the Theoretical Physics with the Differential Geometry and
the Theory of PDEs. This branch of mathematical-physics is characterized, on
the one hand (from the point of view of physicists), by the geometrical
quantization of the covariant Hamiltonian field theories, and, on the other
hand (from the point of view of geometers), by the geometrization of
ordinary Lagrangians and Hamiltonians from Analytical Mechanics or of
multi-time Lagrangians and Hamiltonians used in Theoretical Physics.

In order to reach their aim, the physicists use in their studies that
so-called the \textit{covariant Hamiltonian geometry of physical fields},
which is the \textit{multi-parameter}, or \textit{multi-time}, extension of
the classical Hamiltonian formulation from Mechanics.

It is important to point out that the covariant Hamiltonian geometry of
physical fields appears in the literature of specialty in three distinct
variants (the \textit{multisymplectic} geometry, the \textit{polysymplectic}
geometry and the \textit{De Donder-Weyl covariant Hamiltonian} geometry),
which differ by the \textit{phase space} used in study. All these three
different alternative extensions of the Hamiltonian formulation to field
theory (these extensions originate from the calculus of variations of
multiple integrals) reduce to the classical Hamiltonian formalism from
Mechanics if the number of space-time dimensions equals to one.

Using the technics of the symplectic geometry, the \textit{multisymplectic
covariant geometry of physical fields }is developed by Gotay, Isenberg,
Marsden, Montgomery and their co-workers [10], [11] on a finite-dimensional
multisymplectic phase space.

In order to quantize the covariant Hamiltonian field theory (this is the
final purpose in the framework of quantum field theory), Giachetta,
Mangiarotti and Sardanashvily [8], [9] develop the \textit{polysymplectic
Hamiltonian geometry}, which studies the relations between the equations of
first order Lagrangian field theory on fiber bundles and the covariant
Hamilton equations on a finite-dimensional polysymplectic phase space.

Another convenient approach for quantization of the Hamiltonian field theory
is the \textit{De Donder-Weyl Hamiltonian canonical formulation of field
theory} (this is known for about 80 years), which is intensively studied by
Kanatchikov (please see [12], [13], [14] and references therein) as opposed
to the conventional field-theoretical Hamiltonian formalism, which requires
the space $+$ time decomposition and leads to the picture of a field as a
mechanical system with infinitely many degrees of freedom. The De
Donder-Weyl Hamiltonian approach is achieved by assigning the canonical
momentum like variables to the whole set of space-time derivatives of a
field, that is $\partial _{b}x^{j}\rightarrow p_{j}^{b}$, where $x^{i}$
denote field variables $(i=1,...,n)$, $t^{a}$ are space-time variables $%
(a=1,...,m)$, $\partial _{b}x^{j}$ are space-time derivatives (or first
jets) of field variables and $p_{j}^{b}$ denote \textit{polymomenta}.

In this direction, in the De Donder-Weyl polymomentum canonical theory,
given a Lagrangian function $L=L(t^{a},x^{i},\partial _{b}x^{j})$, the
polymomenta are introduced by the formula $p_{j}^{b}:=\partial L/\partial
(\partial _{b}x^{j})$ and the corresponding \textit{De Donder-Weyl
Hamiltonian function} is given by $H:=\left( \partial _{a}x^{i}\right)
p_{i}^{a}-L$, where it is obvious that $H$ is a function of variables $%
z^{I}:=(t^{a},x^{i},p_{i}^{a})$. In these variables the Euler-Lagrange field
equations can be equivalent rewritten in the form of \textit{De Donder-Weyl
Hamiltonian field equations} [12], [13]%
\begin{equation*}
\begin{array}{cc}
\dfrac{\partial x^{i}}{\partial t^{a}}=\dfrac{\partial H}{\partial p_{i}^{a}}%
, & \dfrac{\partial p_{i}^{a}}{\partial t^{a}}=-\dfrac{\partial H}{\partial
x^{i}},%
\end{array}%
\end{equation*}%
which, for $m=1,$ reduce to the standard Hamilton-Jacobi equations from
Mechanics, and for $m>1$, provide a multi-time covariant generalization of
the Hamiltonian formalism.

From the perspective of geometers, we point out that, following the
geometrical ideas initially stated by Asanov in the paper [4], a \textit{%
multi-time Lagrange contravariant geometry on 1-jet spaces} (in the sense of
distinguished connections, torsions and curvatures) was recently developed
by Neagu and Udri\c{s}te [19], [21]. This geometrical theory is a natural
multi-time extension on 1-jet spaces of the already classical \textit{%
Lagrange geometrical theory of the tangent bundle} elaborated by Miron and
Anastasiei [16]. On the other hand, suggested by the field theoretical
extension of the basic structures of classical Analytical Mechanics within
the framework of the De Donder-Weyl covariant Hamiltonian formulation, the
studies of Miron [15], Atanasiu [5], [6] and others led to the development
of the \textit{Hamilton geometry of the cotangent bundle} exposed in the
book [17].

In such a physical and geometrical context, the aim of this paper is to
expose the basic geometrical objects (such as the distinguished connections,
torsions and curvatures) on the dual 1-jet vector bundle $J^{1\ast }(%
\mathcal{T},M)$ (the polymomentum phase space), coordinated by $%
(t^{a},x^{i},p_{i}^{a})$, where $a=\overline{1,m}$ and $i=\overline{1,n}$.
This geometrical theory (which finally represents a geometrization for
multi-time Hamiltonian functions) is called by us the \textit{multi-time
covariant Hamilton geometry}. Note that the multi-time covariant Hamilton
geometry is a natural multi-time generalization of the Hamilton geometry of
the cotangent bundle [17]. We sincerely hope that the geometrical results
exposed in this paper to have a physical meaning for physicists, physical
meaning which we do not know yet.

Finally, we would like to point out that the \textit{multi-time Legendre jet
duality }between this multi-time covariant Hamilton geometry and the already
constructed multi-time contravariant Lagrange geometry [19] (this is a
geometrization for jet multi-time Lagrangian functions, which is different
of the geometrization from the paper [18]) is a part of our work in progress
and represents a general direction for our future studies.

\section{The dual $1$-jet bundle $J^{1\ast }\left( \mathcal{T},M\right) $}

\hspace{4mm} Let $\mathcal{T}$ and $M$ be two smooth manifolds of dimensions 
$m,$ respectively $n,$ whose local coordinates are $\left( t^{a}\right) _{a=%
\overline{1,m}},$ respectively $\left( x^{i}\right) _{i=\overline{1,n}}.$

\begin{remark}
i) Throughout this paper all geometrical objects and all mappings are
considered of class $C^{\infty }$, expressed by the words \textbf{%
differentiable} or \textbf{smooth}.

ii) Note that the indices $a,b,c,d,e,f,g,h$ run over the set $\left\{
1,2,...,m\right\} $, the indices $i,j,k,l,p,q,r,s$ run over the set $\left\{
1,2,...,n\right\} $, and the Einstein convention of summation is adopted all
over this work.
\end{remark}

Let us consider the $1$-jet vector bundle $J^{1}\left( \mathcal{T},M\right)
\rightarrow \mathcal{T}\times M$ which has the coordinates $%
(t^{a},x^{i},x_{a}^{i})$ and the fibre type $\mathbb{R}^{mn}.$

\begin{remark}
From a physical point of view, we use the following terminology:

\begin{itemize}
\item The manifold $\mathcal{T}$ is regarded as a \textbf{temporal}
manifold, or a \textbf{multi-time} manifold. Note that our temporal
coordinates $t^{a}$ are like the space-time variables in the De Donder-Weyl
covariant Hamilton geometry [12], [14].

\item The manifold $M$ is regarded as a \textbf{spatial} one. Note that our
spatial coordinates $x^{i}$ denote the field variables in the De Donder-Weyl
covariant Hamilton geometry [12], [13], [14].

\item The coordinates $x_{a}^{i}$ are regarded by us as \textbf{partial
directions} or \textbf{partial velocities}. Note that, in the De Donder-Weyl
covariant Hamilton geometry [12], [13], [14], the coordinates $x_{a}^{i}$
are space-time derivatives (or first jets) of the field variables.

\item The fibre bundle $J^{1}\left( \mathcal{T},M\right) \rightarrow 
\mathcal{T}\times M$ is regarded as a \textbf{bundle of configurations}.
This is because in the particular case $\mathcal{T}=\mathbb{R}$ (i.e., the
temporal manifold $\mathcal{T}$ coincides with the usual time axis
represented by the set of real numbers $\mathbb{R}$), we recover the bundle
of configurations which characterizes the classical non-autonomous
mechanics. For more details, please see Abraham and Marsden [1] and Buc\u{a}%
taru and Miron [7].
\end{itemize}
\end{remark}

In order to simplify the notations, we will use the notation $E=J^{1}\left( 
\mathcal{T},M\right) $. We recall that the transformation of coordinates $%
\left( t^{a},x^{i},x_{a}^{i}\right) \longleftrightarrow \left( \tilde{t}^{a},%
\tilde{x}^{i},\tilde{x}_{a}^{i}\right) $, induced from $\mathcal{T}\times M$
on the $1$-jet space $E$, are given by%
\begin{equation}
\left\{ 
\begin{array}{ll}
\tilde{t}^{a}=\tilde{t}^{a}\left( t^{b}\right) , & \det \left( \dfrac{%
\partial \tilde{t}^{a}}{\partial t^{b}}\right) \neq 0,\medskip \\ 
\tilde{x}^{i}=\tilde{x}^{i}\left( x^{j}\right) , & \det \left( \dfrac{%
\partial \tilde{x}^{i}}{\partial x^{j}}\right) \neq 0,\medskip \\ 
\tilde{x}_{a}^{i}=\dfrac{\partial \tilde{x}^{i}}{\partial x^{j}}\dfrac{%
\partial t^{b}}{\partial \tilde{t}^{a}}x_{b}^{j}. & 
\end{array}%
\right.  \label{tr-group-jet}
\end{equation}

\begin{remark}
Let $\mathbb{R}\times TM$ be the trivial bundle over the base tangent space $%
TM,$ whose coordinates induced by $TM$ are $\left( t,x^{i},y^{i}\right) $, $%
i=\overline{1,n},$ $n=\dim M,$ $t\in \mathbb{R}$ being a temporal parameter.
Then, the changes of coordinates on the trivial bundle $\mathbb{R}\times
TM\rightarrow TM$ are given by%
\begin{equation*}
\left\{ 
\begin{array}{l}
\tilde{t}=t,\medskip \\ 
\tilde{x}^{i}=\tilde{x}^{i}\left( x^{j}\right) ,\medskip \\ 
\tilde{y}^{i}=\dfrac{\partial \tilde{x}^{i}}{\partial x^{j}}y^{j}.%
\end{array}%
\right.
\end{equation*}

A time dependent Lagrangian function for $M$ is a real valued function $L$
on $\mathbf{E}=\mathbb{R}\times TM.$ Such Lagrangians (called \textbf{%
rheonomic} or \textbf{non-autonomous}) are important for variational
calculus and non-autonomous mechanics. A \textbf{rheonomic Lagrange geometry}
on the trivial bundle $\mathbf{E}=\mathbb{R}\times TM$ was developed by
Anastasiei, Kawaguchi and Miron [2], [3], [16].

More general, let us consider the 1-jet bundle $\mathbb{E}=J^{1}(\mathbb{R}%
,M)\equiv \mathbb{R}\times TM$ over the product manifold $\mathbb{R}\times M$%
. Then, the changes of coordinates on $\mathbb{E}$ have the form%
\begin{equation*}
\left\{ 
\begin{array}{l}
\tilde{t}=\tilde{t}\left( t\right) ,\medskip  \\ 
\tilde{x}^{i}=\tilde{x}^{i}\left( x^{j}\right) ,\medskip  \\ 
\tilde{y}^{i}=\dfrac{\partial \tilde{x}^{i}}{\partial x^{j}}\dfrac{dt}{d%
\tilde{t}}y^{j}.%
\end{array}%
\right. 
\end{equation*}%
These transformations of coordinates point out the \textbf{relativistic}
character played by the time $t$ in the non-autonomous mechanics. A \textbf{%
relativistic rheonomic Lagrange geometry} on the 1-jet bundle $\mathbb{E}%
=J^{1}(\mathbb{R},M)$ was developed by Neagu in [20]. This geometry has many
similarities with the geometry elaborated by Anastasiei, Kawaguchi and
Miron, but they are however distinct ones.
\end{remark}

Using the general theory of vector bundles, we can consider the dual $1$-jet
vector bundle $E^{\ast }=J^{1\ast }\left( \mathcal{T},M\right) $, whose
fibre type is also $\mathbb{R}^{mn}.$ In fact, the construction of the dual
1-jet vector bundle relies on the substitution of coordinates $%
x_{a}^{i}\rightarrow p_{i}^{a}$, that is if $E=J^{1}\left( \mathcal{T}%
,M\right) $ is coordinated by $\left( t^{a},x^{i},x_{a}^{i}\right) $, then
the corresponding dual $1$-jet space $E^{\ast }=J^{1\ast }\left( \mathcal{T}%
,M\right) $ is coordinated by $\left( t^{a},x^{i},p_{i}^{a}\right) .$
Moreover, the transformation of coordinates $\left(
t^{a},x^{i},p_{i}^{a}\right) \leftrightarrow \left( \tilde{t}^{a},\tilde{x}%
^{i},\tilde{p}_{i}^{a}\right) $, induced from $\mathcal{T}\times M$ on the
dual $1$-jet space $E^{\ast },$ are given by%
\begin{equation}
\left\{ 
\begin{array}{ll}
\tilde{t}^{a}=\tilde{t}^{a}\left( t^{b}\right) , & \det \left( \dfrac{%
\partial \tilde{t}^{a}}{\partial t^{b}}\right) \neq 0,\medskip \\ 
\tilde{x}^{i}=\tilde{x}^{i}\left( x^{j}\right) , & \det \left( \dfrac{%
\partial \tilde{x}^{i}}{\partial x^{j}}\right) \neq 0,\medskip \\ 
\tilde{p}_{i}^{a}=\dfrac{\partial x^{j}}{\partial \tilde{x}^{i}}\dfrac{%
\partial \tilde{t}^{a}}{\partial t^{b}}p_{j}^{b}. & 
\end{array}%
\right.  \label{tr-group-dual-jet}
\end{equation}

According to Kanatchikov's physical terminology [13], we introduce

\begin{definition}
The coordinates $p_{i}^{a},$ $a=\overline{1,m},$ $i=\overline{1,n},$ are
called \textbf{polymomenta}, and\textbf{\ }the dual $1$-jet space $E^{\ast
}\ $is called the \textbf{polymomentum phase space}.
\end{definition}

\begin{remark}
Let $\mathbb{R}\times T^{\ast }M$ be the trivial bundle over the base
cotangent space $T^{\ast }M$, whose coordinates induced by $T^{\ast }M$ are $%
\left( t,x^{i},p_{i}\right) ,$ $i=\overline{1,n},$ $n=\dim M,$ $t\in \mathbb{%
R}$ being temporal parameter. Then, the changes of coordinates on the
trivial bundle$\ \mathbb{R}\times T^{\ast }M\rightarrow T^{\ast }M$ are
given by\ 
\begin{equation*}
\left\{ 
\begin{array}{l}
\tilde{t}=t\medskip \\ 
\tilde{x}^{i}=\tilde{x}^{i}\left( x^{j}\right) ,\medskip \\ 
\tilde{p}_{i}=\dfrac{\partial x^{j}}{\partial \tilde{x}^{i}}p_{j}%
\end{array}%
\right.
\end{equation*}

A time dependent Hamiltonian function for $M$ is a real valued function $H$
on $\mathbf{E}^{\ast }=\mathbb{R}\times T^{\ast }M.$ Such Hamiltonians
(called \textbf{rheonomic} or \textbf{non-autonomous}) are important for
covariant Hamiltonian approach of non-autonomous mechanics. A geometrization
of these Hamiltonians on the trivial bundle $\mathbf{E}^{\ast }=\mathbb{R}%
\times T^{\ast }M$ is studied by Miron, Atanasiu and their co-workers [5],
[6], [17].

More general, in our jet geometrical approach, we use a relativistic time $t$%
. This is because we have a change of coordinates (induced on $\mathbf{E}%
^{\ast }=\mathbb{R}\times T^{\ast }M$ by the product manifold $\mathbb{R}%
\times M$) of the form%
\begin{equation*}
\left\{ 
\begin{array}{l}
\tilde{t}=\tilde{t}\left( t\right) \medskip \\ 
\tilde{x}^{i}=\tilde{x}^{i}\left( x^{j}\right) ,\medskip \\ 
\tilde{p}_{i}=\dfrac{\partial x^{j}}{\partial \tilde{x}^{i}}\dfrac{d\tilde{t}%
}{dt}p_{j}.%
\end{array}%
\right.
\end{equation*}

The above two groups of transformations of coordinates are in use for the 
\textbf{Hamilton geometry} that study the metrical structure 
\begin{equation*}
g^{ij}\left( t,x,p\right) =\frac{1}{2}\frac{\partial ^{2}H}{\partial
p_{i}\partial p_{j}}.
\end{equation*}
\end{remark}

Doing a transformation of coordinates (\ref{tr-group-dual-jet}) on $E^{\ast
} $, we find the following results:

\begin{proposition}
The local natural basis $\left\{ \dfrac{\partial }{\partial t^{a}},\dfrac{%
\partial }{\partial x^{i}},\dfrac{\partial }{\partial p_{i}^{a}}\right\} $
of the Lie algebra of the vector fields $\chi \left( E^{\ast }\right) $
transforms by the laws:%
\begin{equation}
\begin{array}{l}
\dfrac{\partial }{\partial t^{a}}=\dfrac{\partial \tilde{t}^{b}}{\partial
t^{a}}\dfrac{\partial }{\partial \tilde{t}^{b}}+\dfrac{\partial \tilde{p}%
_{j}^{b}}{\partial t^{a}}\dfrac{\partial }{\partial \tilde{p}_{j}^{b}}%
,\medskip \\ 
\dfrac{\partial }{\partial x^{i}}=\dfrac{\partial \tilde{x}^{j}}{\partial
x^{i}}\dfrac{\partial }{\partial \tilde{x}^{j}}+\dfrac{\partial \tilde{p}%
_{j}^{b}}{\partial x^{i}}\dfrac{\partial }{\partial \tilde{p}_{j}^{b}}%
,\medskip \\ 
\dfrac{\partial }{\partial p_{i}^{a}}=\dfrac{\partial \tilde{t}^{b}}{%
\partial t^{a}}\dfrac{\partial x^{i}}{\partial \tilde{x}^{j}}\dfrac{\partial 
}{\partial \tilde{p}_{j}^{b}};%
\end{array}
\label{tr-laws-basis}
\end{equation}
\end{proposition}

\begin{proposition}
The local natural cobasis $\left\{ dt^{a},dx^{i},dp_{i}^{a}\right\} $ of the
algebra Lie of the covector fields $\chi ^{\ast }\left( E^{\ast }\right) $
transforms by the laws:%
\begin{equation}
\begin{array}{l}
dt^{a}=\dfrac{\partial t^{a}}{\partial \tilde{t}^{b}}d\tilde{t}^{b},\medskip
\\ 
dx^{i}=\dfrac{\partial x^{i}}{\partial \tilde{x}^{j}}d\tilde{x}^{j},\medskip
\\ 
dp_{i}^{a}=\dfrac{\partial p_{i}^{a}}{\partial \tilde{t}^{b}}d\tilde{t}^{b}+%
\dfrac{\partial p_{i}^{a}}{\partial \tilde{x}^{j}}d\tilde{x}^{j}+\dfrac{%
\partial t^{a}}{\partial \tilde{t}^{b}}\dfrac{\partial \tilde{x}^{j}}{%
\partial x^{i}}d\tilde{p}_{j}^{b}.%
\end{array}
\label{tr-laws-cobasis}
\end{equation}
\end{proposition}

\section{Nonlinear connections}

\hspace{4mm} Taking into account the complicated transformation rules (\ref%
{tr-laws-basis}) and (\ref{tr-laws-cobasis}), we need a \textit{nonlinear
connection} on the dual 1-jet space $E^{\ast }$, in order to construct some 
\textit{adapted bases} whose transformation rules to be more simple
(tensorial ones, for instance).

Let $u^{\ast }=\left( t^{a},x^{i},p_{i}^{a}\right) \in E^{\ast }$ be an
arbitrary point and let us consider the differential map%
\begin{equation*}
\pi ^{\ast }\ _{\ast ,u^{\ast }}:T_{u^{\ast }}E^{\ast }\rightarrow T_{\left(
t,x\right) }\left( \mathcal{T}\times M\right)
\end{equation*}%
of the canonical projection $\pi ^{\ast }:E^{\ast }\rightarrow \mathcal{T}%
\times M,$ $\pi ^{\ast }\left( u^{\ast }\right) =\left( t,x\right) $,
together with its vector subspace $W_{u^{\ast }}=Ker$ $\pi ^{\ast }\ _{\ast
,u^{\ast }}\subset T_{u^{\ast }}E^{\ast }.\ $Because the differential map $%
\pi ^{\ast }\ _{\ast ,u^{\ast }}$ is a surjection, we find that we have $%
\dim _{\mathbb{R}}W_{u^{\ast }}=mn$ and, moreover, a basis in $W_{u^{\ast }%
\text{ }}$is determined by $\left\{ \dfrac{\partial }{\partial p_{i}^{a}}%
\right\} .$

So, the map%
\begin{equation*}
\mathcal{W}:u^{\ast }\in E^{\ast }\rightarrow W_{u^{\ast }}\subset
T_{u^{\ast }}E^{\ast }
\end{equation*}%
is a differential distribution which is called the \textit{vertical
distribution} on the dual 1-jet vector bundle $E^{\ast }.$

\begin{definition}
A differential distribution $\mathcal{H}:u^{\ast }\in E^{\ast }\rightarrow
H_{u^{\ast }}\subset T_{u^{\ast }}E^{\ast }$, which is supplementary to the
vertical distribution $\mathcal{W},$ i. e.%
\begin{equation}
T_{u^{\ast }}E^{\ast }=H_{u^{\ast }}\oplus W_{u^{\ast },}\text{ }\forall 
\text{ }u^{\ast }\in E^{\ast },  \label{hor-distrib}
\end{equation}%
is called a \textbf{nonlinear connection} on $E^{\ast }$ ($\mathcal{H}\ $is
also called the \textbf{horizontal distribution} on $E^{\ast }$).
\end{definition}

The above definition implies that $\dim _{\mathbb{R}}H_{u^{\ast }}=m+n,$ $%
\forall $ $u^{\ast }\in E^{\ast }$, and that the Lie algebra of the vector
fields $\chi \left( E^{\ast }\right) $ can be decomposed in the direct sum\ $%
\chi \left( E^{\ast }\right) =\mathcal{S}\left( \mathcal{H}\right) \oplus 
\mathcal{S}\left( \mathcal{W}\right) ,$ where $\mathcal{S}\left( \mathcal{H}%
\right) $ (resp. $\mathcal{S}\left( \mathcal{W}\right) $) is the set of
differentiable sections on $\mathcal{H}$ (resp. $\mathcal{W}$)$.$

Supposing that $\mathcal{H}$\ is a fixed nonlinear connection on $E^{\ast }$%
, we have the isomorphism%
\begin{equation*}
\left. \pi ^{\ast }\ _{\ast ,u^{\ast }}\right\vert _{H_{u^{\ast
}}}:H_{u^{\ast }}\rightarrow T_{\pi ^{\ast }\left( u^{\ast }\right) }\left( 
\mathcal{T}\times M\right) ,
\end{equation*}%
which allows us to prove the following result:

\begin{theorem}
\label{Th-nlc} i) There exist the unique, linear independent, horizontal
vector fields $\dfrac{\delta }{\delta t^{a}},$ $\dfrac{\delta }{\delta x^{i}}%
\in \mathcal{S}\left( \mathcal{H}\right) $, having the properties:%
\begin{equation}
\begin{array}{cc}
\pi ^{\ast }\ _{\ast }\left( \dfrac{\delta }{\delta t^{a}}\right) =\dfrac{%
\partial }{\partial t^{a}}, & \pi ^{\ast }\ _{\ast }\left( \dfrac{\delta }{%
\delta x^{i}}\right) =\dfrac{\partial }{\partial x^{i}}.%
\end{array}
\label{delta-t-si-x}
\end{equation}

ii) The horizontal vector fields $\dfrac{\delta }{\delta t^{a}}$ and $\dfrac{%
\delta }{\delta x^{i}}$ can be uniquely written in the form:%
\begin{equation}
\begin{array}{cc}
\dfrac{\delta }{\delta t^{a}}=\dfrac{\partial }{\partial t^{a}}-\underset{1}{%
N}\underset{}{\overset{\left( b\right) }{_{\left( j\right) a}}}\dfrac{%
\partial }{\partial p_{j}^{b}}, & \dfrac{\delta }{\delta x^{i}}=\dfrac{%
\partial }{\partial x^{i}}-\underset{2}{N}\underset{}{\overset{\left(
b\right) }{_{\left( j\right) i}}}\dfrac{\partial }{\partial p_{j}^{b}}.%
\end{array}
\label{form-of-delta-t-si-x}
\end{equation}

iii) With respect to a transformation of coordinates (\ref{tr-group-dual-jet}%
) on $E^{\ast }$, the coefficients \ $\underset{1}{N}\underset{}{\overset{%
\left( b\right) }{_{\left( j\right) a}}}$\ and $\underset{2}{N}\underset{}{%
\overset{\left( b\right) }{_{\left( j\right) i}}}$\ obey the rules%
\begin{equation}
\begin{array}{l}
\underset{1}{\tilde{N}}\underset{}{\overset{\left( b\right) }{_{\left(
j\right) c}}}\dfrac{\partial \tilde{t}^{c}}{\partial t^{a}}=\underset{1}{N}%
\underset{}{\overset{\left( c\right) }{_{\left( k\right) a}}}\dfrac{\partial 
\tilde{t}^{b}}{\partial t^{c}}\dfrac{\partial x^{k}}{\partial \tilde{x}^{j}}-%
\dfrac{\partial \tilde{p}_{j}^{b}}{\partial t^{a}},\medskip \\ 
\underset{2}{\tilde{N}}\underset{}{\overset{\left( b\right) }{_{\left(
j\right) k}}\dfrac{\partial \tilde{x}^{k}}{\partial x^{i}}}=\underset{2}{N}%
\underset{}{\overset{\left( c\right) }{_{\left( k\right) i}}}\dfrac{\partial 
\tilde{t}^{b}}{\partial t^{c}}\dfrac{\partial x^{k}}{\partial \tilde{x}^{j}}-%
\dfrac{\partial \tilde{p}_{j}^{b}}{\partial x^{i}}.%
\end{array}
\label{tr-rules-nlc}
\end{equation}

iv) To give a nonlinear connection $\mathcal{H}$\ on $E^{\ast }$\ is
equivalent to give a set of local functions $N=\left( \underset{1}{N}%
\underset{}{\overset{\left( b\right) }{_{\left( j\right) a}}},\ \underset{2}{%
N}\underset{}{\overset{\left( b\right) }{_{\left( j\right) i}}}\right) \ $%
which transform by the rules (\ref{tr-rules-nlc}).
\end{theorem}

\begin{proof}
Let $\dfrac{\delta }{\delta t^{a}},$ $\dfrac{\delta }{\delta x^{i}}\in \chi
\left( E^{\ast }\right) $ be vector fields on $E^{\ast }$, locally expressed
by%
\begin{equation*}
\begin{array}{l}
\dfrac{\delta }{\delta t^{a}}=A_{a}^{b}\dfrac{\partial }{\partial t^{b}}%
+A_{a}^{j}\dfrac{\partial }{\partial x^{j}}+A_{\left( j\right) a}^{\left(
b\right) }\dfrac{\partial }{\partial p_{j}^{b}},\medskip \\ 
\dfrac{\delta }{\delta x^{i}}=X_{i}^{b}\dfrac{\partial }{\partial t^{b}}%
+X_{i}^{j}\dfrac{\partial }{\partial x^{j}}+X_{\left( j\right) i}^{\left(
b\right) }\dfrac{\partial }{\partial p_{j}^{b}},%
\end{array}%
\end{equation*}%
which verify the relations (\ref{delta-t-si-x}). Then, taking into account
the local expression of the map $\pi ^{\ast }{}_{\ast }$, we get%
\begin{eqnarray*}
A_{a}^{b} &=&\delta _{a}^{b},\ A_{a}^{j}=0,\ A_{\left( j\right) a}^{\left(
b\right) }=-\underset{1}{N}\underset{}{\overset{\left( b\right) }{_{\left(
j\right) a}}}, \\
X_{i}^{b} &=&0,\ X_{i}^{j}=\delta _{i}^{j},\ X_{\left( j\right) i}^{\left(
b\right) }=-\underset{2}{N}\underset{}{\overset{\left( b\right) }{_{\left(
j\right) i}}}.
\end{eqnarray*}%
These equalities prove the form (\ref{form-of-delta-t-si-x}) of the vector
fields from Theorem \ref{Th-nlc}, together with their linear independence.
The uniqueness of the coefficients $\underset{1}{N}\underset{}{\overset{%
\left( b\right) }{_{\left( j\right) a}}}\ $and $\underset{2}{N}\underset{}{%
\overset{\left( b\right) }{_{\left( j\right) i}}}$ is obvious.

Because the vector fields $\dfrac{\delta }{\delta t^{a}}$ and $\dfrac{\delta 
}{\delta x^{i}}$ are globally defined, we deduce that a change of
coordinates (\ref{tr-group-dual-jet}) on $E^{\ast }$ produces a
transformation of the coefficients $\underset{1}{N}\underset{}{\overset{%
\left( b\right) }{_{\left( j\right) a}}}$ and $\underset{2}{N}\underset{}{%
\overset{\left( b\right) }{_{\left( j\right) i}}}$ by the rules (\ref%
{tr-rules-nlc}).

Finally, starting with a set of functions $N=\left( \underset{1}{N}\underset{%
}{\overset{\left( b\right) }{_{\left( j\right) a}}},\ \underset{2}{N}%
\underset{}{\overset{\left( b\right) }{_{\left( j\right) i}}}\right) $ which
respect the rules (\ref{tr-rules-nlc})$,$ we can construct the horizontal
distribution\ $\mathcal{H}$\ taking%
\begin{equation*}
H_{u^{\ast }}=Span\left\{ \left. \frac{\delta }{\delta t^{a}}\right\vert
_{u^{\ast }},\left. \frac{\delta }{\delta x^{i}}\right\vert _{u^{\ast
}}\right\} .
\end{equation*}%
The decomposition $T_{u^{\ast }}E^{\ast }=H_{u^{\ast }}\oplus W_{u^{\ast }}$
is obvious now.
\end{proof}

\begin{definition}
i) The geometrical entity$\ \underset{1}{N}=\left( \underset{1}{N}\underset{}%
{\overset{\left( b\right) }{_{\left( j\right) a}}}\right) \ $(resp. $%
\underset{2}{N}=\left( \underset{2}{N}\underset{}{\overset{\left( b\right) }{%
_{\left( j\right) i}}}\right) $) is called a \textbf{temporal} (resp. 
\textbf{spatial}) \textbf{nonlinear connection} on $E^{\ast }.$

ii) The set of the linear independent vector fields%
\begin{equation}
\left\{ \frac{\delta }{\delta t^{a}},\frac{\delta }{\delta x^{i}},\frac{%
\partial }{\partial p_{i}^{a}}\right\} \subset \chi \left( E^{\ast }\right)
\label{ad-basis-nlc}
\end{equation}%
is called the \textbf{adapted basis of vector fields attached to the
nonlinear connection} $N=\left( \underset{1}{N},\underset{2}{N}\right) $.
\end{definition}

\begin{example}
Let us consider that $h_{ab}(t)\ $(resp. $\varphi _{ij}(x)$) is a
semi-Riemannian metric on $\mathcal{T\ }$(resp. $M$) and let $\varkappa
_{bc}^{a}(t)\ $(resp. $\gamma _{jk}^{i}(x)$) be its Christoffel symbols.
Setting%
\begin{equation}
\begin{array}{cc}
\overset{0}{\underset{1}{N}}\underset{}{\overset{\left( b\right) }{_{\left(
j\right) c}}}=\varkappa _{ac}^{b}p_{j}^{a}, & \overset{0}{\underset{2}{N}}%
\underset{}{\overset{\left( b\right) }{_{\left( j\right) k}}}=-\gamma
_{jk}^{i}p_{i}^{b},%
\end{array}
\label{can-nlc-asoc-to-metric}
\end{equation}%
we obtain that%
\begin{equation*}
\overset{0}{\underset{}{N}}=\left( \overset{0}{\underset{1}{N}}\underset{}{%
\overset{\left( b\right) }{_{\left( j\right) c}}},\text{ }\overset{0}{%
\underset{2}{N}}\underset{}{\overset{\left( b\right) }{_{\left( j\right) k}}}%
\right)
\end{equation*}%
is a nonlinear connection on $E^{\ast }=J^{1\ast }\left( \mathcal{T}%
,M\right) $, which is called the \textbf{canonical nonlinear connection on }$%
E^{\ast }$\textbf{\ attached to the pair of semi-Riemannian metrics }$%
h_{ab}(t)$\textbf{\ and }$\varphi _{ij}(x).$
\end{example}

With respect to the coordinate transformations (\ref{tr-group-dual-jet}) the
adapted basis (\ref{ad-basis-nlc}) has its transformation laws extremely
simple, namely (tensorial ones)%
\begin{equation}
\begin{array}{l}
\dfrac{\delta }{\delta t^{a}}=\dfrac{\partial \tilde{t}^{b}}{\partial t^{a}}%
\dfrac{\delta }{\delta \tilde{t}^{b}},\medskip \\ 
\dfrac{\delta }{\delta x^{i}}=\dfrac{\partial \tilde{x}^{j}}{\partial x^{i}}%
\dfrac{\delta }{\delta \tilde{x}^{j}},\medskip \\ 
\dfrac{\partial }{\partial p_{i}^{a}}=\dfrac{\partial \tilde{t}^{b}}{%
\partial t^{a}}\dfrac{\partial x^{i}}{\partial \tilde{x}^{j}}\dfrac{\partial 
}{\partial \tilde{p}_{j}^{b}},%
\end{array}
\label{tr-laws-ad-basis}
\end{equation}%
in contrast with the transformations (\ref{tr-laws-basis}).

The dual basis of the adapted basis (\ref{ad-basis-nlc}) is given by%
\begin{equation}
\left\{ dt^{a},dx^{i},\delta p_{i}^{a}\right\} \subset \chi ^{\ast }\left(
E^{\ast }\right)  \label{ad-cobasis-nlc}
\end{equation}%
where%
\begin{equation}
\delta p_{i}^{a}=dp_{i}^{a}+\underset{1}{N}\overset{(a)}{\underset{}{%
_{\left( i\right) b}}}dt^{b}+\underset{2}{N}\overset{(a)}{\underset{}{%
_{\left( i\right) j}}}dx^{j}.  \label{expr-delta-p}
\end{equation}

\begin{definition}
The dual basis of covector fields given by (\ref{ad-cobasis-nlc}) and (\ref%
{expr-delta-p}) is called the \textbf{adapted cobasis of covector fields
attached to the nonlinear connection} $N=\left( \underset{1}{N},\underset{2}{%
N}\right) $.
\end{definition}

It is obvious that we always have%
\begin{equation*}
\begin{array}{lll}
\dfrac{\delta }{\delta t^{a}}\text{ }\rfloor \ dt^{b}=\delta _{a}^{b}, & 
\dfrac{\delta }{\delta t^{a}}\text{ }\rfloor \ dx^{j}=0, & \dfrac{\delta }{%
\delta t^{a}}\text{ }\rfloor \ \delta p_{j}^{b}=0,\medskip \\ 
\dfrac{\delta }{\delta x^{i}}\text{ }\rfloor \ dt^{b}=0, & \dfrac{\delta }{%
\delta x^{i}}\text{ }\rfloor \ dx^{j}=\delta _{i}^{j}, & \dfrac{\delta }{%
\delta x^{i}}\text{ }\rfloor \ \delta p_{j}^{b}=0,\medskip \\ 
\dfrac{\partial }{\partial p_{i}^{a}}\text{ }\rfloor \ dt^{b}=0, & \dfrac{%
\partial }{\partial p_{i}^{a}}\text{ }\rfloor \ dx^{j}=0, & \dfrac{\partial 
}{\partial p_{i}^{a}}\text{ }\rfloor \ \delta p_{j}^{b}=\delta
_{a}^{b}\delta _{j}^{i}.\medskip%
\end{array}%
\end{equation*}%
Moreover, with respect to (\ref{tr-group-dual-jet}), we obtain the following
tensorial transformation rules:%
\begin{equation}
\begin{array}{l}
dt^{a}=\dfrac{\partial t^{a}}{\partial \tilde{t}^{b}}d\tilde{t}^{b},\medskip
\\ 
dx^{i}=\dfrac{\partial x^{i}}{\partial \tilde{x}^{j}}d\tilde{x}^{j},\medskip
\\ 
\delta p_{i}^{a}=\dfrac{\partial t^{a}}{\partial \tilde{t}^{b}}\dfrac{%
\partial \tilde{x}^{j}}{\partial x^{i}}\delta \tilde{p}_{j}^{b}.%
\end{array}
\label{tr-laws-ad-cobasis}
\end{equation}

As a consequence of the preceding assertions, we find the following simple
result:

\begin{proposition}
i) The Lie algebra $\chi \left( E^{\ast }\right) $ of the vector fields on $%
E^{\ast }$ decomposes in the direct sum%
\begin{equation*}
\chi \left( E^{\ast }\right) =\chi \left( \mathcal{H}_{\mathcal{T}}\right)
\oplus \chi \left( \mathcal{H}_{M}\right) \oplus \chi \left( \mathcal{W}%
\right) ,
\end{equation*}%
where%
\begin{equation*}
\chi \left( \mathcal{H}_{\mathcal{T}}\right) =Span\left\{ \frac{\delta }{%
\delta t^{a}}\right\} ,\text{ }\chi \left( \mathcal{H}_{M}\right)
=Span\left\{ \frac{\delta }{\delta x^{i}}\right\} ,\text{ }\chi \left( 
\mathcal{W}\right) =Span\left\{ \frac{\partial }{\partial p_{i}^{a}}\right\}
.
\end{equation*}

ii) The Lie algebra $\chi ^{\ast }\left( E^{\ast }\right) $ of the covector
fields on $E^{\ast }$ decomposes in the direct sum%
\begin{equation*}
\chi ^{\ast }\left( E^{\ast }\right) =\chi ^{\ast }\left( \mathcal{H}_{%
\mathcal{T}}\right) \oplus \chi ^{\ast }\left( \mathcal{H}_{M}\right) \oplus
\chi ^{\ast }\left( \mathcal{W}\right) ,
\end{equation*}%
where%
\begin{equation*}
\chi ^{\ast }\left( \mathcal{H}_{\mathcal{T}}\right) =Span\left\{
dt^{a}\right\} ,\text{ }\chi ^{\ast }\left( \mathcal{H}_{M}\right)
=Span\left\{ dx^{i}\right\} ,\text{ }\chi ^{\ast }\left( \mathcal{W}\right)
=Span\left\{ \delta p_{i}^{a}\right\} .
\end{equation*}
\end{proposition}

\begin{definition}
The distribution $\mathcal{H}_{\mathcal{T}}$ (resp. $\mathcal{H}_{M}$) is
called the \textbf{$\mathcal{T}$-horizontal distribution} (resp. $M$\textbf{%
-horizontal distribution}) on $E^{\ast }$.
\end{definition}

Let $h_{\mathcal{T}},$ $h_{M}$ and $w$ be the $\mathcal{T}$- horizontal, $M$%
-horizontal and vertical, respectively, canonical projections associated to
the decomposition of vector fields on $E^{\ast }$. Then, it is obvious that
we have the relations:%
\begin{equation}
\begin{array}{c}
h_{\mathcal{T}}+h_{M}+w=I,\ h_{\mathcal{T}}^{2}=h_{\mathcal{T}},\
h_{M}^{2}=h_{M},\ w^{2}=w,\medskip \\ 
h_{\mathcal{T}}\circ h_{M}=h_{M}\circ h_{\mathcal{T}}=0,\text{ }h_{\mathcal{T%
}}\circ w=w\circ h_{\mathcal{T}}=0,\medskip \\ 
h_{M}\circ w=w\circ h_{M}=0.%
\end{array}
\label{projections}
\end{equation}

Starting with a vector field $X\in \chi \left( E^{\ast }\right) $, we denote%
\begin{equation*}
X^{\mathcal{H}_{\mathcal{T}}}=h_{\mathcal{T}}X,\text{ }X^{\mathcal{H}%
_{M}}=h_{M}X,\text{ }X^{\mathcal{W}}=wX.
\end{equation*}%
Therefore, we have the unique decomposition%
\begin{equation}
X=X^{\mathcal{H}_{\mathcal{T}}}+X^{\mathcal{H}_{M}}+X^{\mathcal{W}},
\label{decomp-vect-fields}
\end{equation}%
where, using the adapted basis (\ref{ad-basis-nlc}), we get%
\begin{equation*}
X^{\mathcal{H}_{\mathcal{T}}}=X^{a}\frac{\delta }{\delta t^{a}},\text{ }X^{%
\mathcal{H}_{M}}=X^{i}\frac{\delta }{\delta x^{i}},\text{ }X^{\mathcal{W}%
}=X_{\left( i\right) }^{\left( a\right) }\frac{\partial }{\partial p_{i}^{a}}%
.
\end{equation*}%
By means of (\ref{tr-laws-ad-basis}), we deduce that%
\begin{equation*}
\tilde{X}^{a}=\frac{\partial \tilde{t}^{a}}{\partial t^{b}}X^{b},\text{ }%
\tilde{X}^{i}=\frac{\partial \tilde{x}^{i}}{\partial x^{j}}X^{j},\text{ }%
\tilde{X}_{\left( i\right) }^{\left( a\right) }=\frac{\partial \tilde{t}^{a}%
}{\partial t^{b}}\frac{\partial x^{j}}{\partial \tilde{x}^{i}}X_{\left(
j\right) }^{\left( b\right) }.
\end{equation*}

Now, we can give the following important results:

\begin{proposition}
\label{H-integrability}The horizontal distribution $\mathcal{H}$ is
integrable if and only if for any vector fields $X,Y\in \chi \left( E^{\ast
}\right) $ we have%
\begin{equation*}
\left[ X^{\mathcal{H}_{\mathcal{T}}},Y^{\mathcal{H}_{\mathcal{T}}}\right] ^{%
\mathcal{W}}=0,\text{ }\left[ X^{\mathcal{H}_{\mathcal{T}}},Y^{\mathcal{H}%
_{M}}\right] ^{\mathcal{W}}=0,\text{ }\left[ X^{\mathcal{H}_{M}},Y^{\mathcal{%
H}_{M}}\right] ^{\mathcal{W}}=0.
\end{equation*}
\end{proposition}

\begin{proof}
Indeed, the Poisson bracket between two $\mathcal{T}$-horizontal vector
fields $X^{\mathcal{H}_{\mathcal{T}}},$ $Y^{\mathcal{H}_{\mathcal{T}}}$
(resp. $M$-horizontal vector fields $X^{\mathcal{H}_{M}},$ $Y^{\mathcal{H}%
_{M}}$) or between a $\mathcal{T}$-horizontal vector field $X^{\mathcal{H}_{%
\mathcal{T}}}$ and an $M$-horizontal vector field $Y^{\mathcal{H}_{M}}$
belongs to the horizontal distribution $\mathcal{H}$ if and only if the
preceding three equations hold good.
\end{proof}

\begin{proposition}
The vertical distribution $\mathcal{W}$ is always integrable.
\end{proposition}

A similar theory can be done for $1$-forms. With respect to the
decomposition of covector fields on $E^{\ast }$, any $1$-form $\omega \in
\chi ^{\ast }\left( E^{\ast }\right) $ can be uniquely written in the form%
\begin{equation}
\omega =\omega ^{\mathcal{H}_{\mathcal{T}}}+\omega ^{\mathcal{H}_{M}}+\omega
^{\mathcal{W}},  \label{decomp-covect-fields}
\end{equation}%
where%
\begin{equation*}
\omega ^{\mathcal{H}_{\mathcal{T}}}=\omega \circ h_{\mathcal{T}},\text{ }%
\omega ^{\mathcal{H}_{M}}=\omega \circ h_{M},\text{ }\omega ^{\mathcal{W}%
}=\omega \circ w.
\end{equation*}%
In the adapted cobasis (\ref{ad-cobasis-nlc}), we have in fact%
\begin{equation*}
\omega =\omega _{a}dt^{a}+\omega _{i}dx^{i}+\omega _{\left( a\right)
}^{\left( i\right) }\delta p_{i}^{a}.
\end{equation*}%
The components $\omega _{a},$ $\omega _{i},$ $\omega _{\left( a\right)
}^{\left( i\right) }$ are transformed by (\ref{tr-group-dual-jet}) as
follows:%
\begin{equation*}
\omega _{a}=\frac{\partial \tilde{t}^{b}}{\partial t^{a}}\tilde{\omega}_{b},%
\text{ }\omega _{i}=\frac{\partial \tilde{x}^{j}}{\partial x^{i}}\tilde{%
\omega}_{j},\text{ }\omega _{\left( a\right) }^{\left( i\right) }=\frac{%
\partial \tilde{t}^{b}}{\partial t^{a}}\frac{\partial x^{i}}{\partial \tilde{%
x}^{j}}\omega _{\left( b\right) }^{\left( j\right) }.
\end{equation*}

\section{The algebra of distinguished tensor fields}

\hspace{4mm} In the study of differential geometry of the dual vector bundle
of $1$-jets $E^{\ast }=J^{1\ast }\left( \mathcal{T},M\right) $, a central
role is played by the \textit{distinguished tensors} or, briefly, $d$\textit{%
-tensors}.

\begin{definition}
A geometrical object $T=\left( T_{bj\left( c\right) \left( l\right)
...}^{ai\left( k\right) \left( d\right) ...}\right) $ on $E^{\ast }$ which,
with respect to a transformation of coordinates (\ref{tr-group-dual-jet}) on 
$E^{\ast }$, verifies the transformation rules%
\begin{equation*}
T_{bj\left( c\right) \left( l\right) ...}^{ai\left( k\right) \left( d\right)
...}=\tilde{T}_{fq\left( g\right) \left( s\right) ...}^{ep\left( r\right)
\left( h\right) ...}\frac{\partial t^{a}}{\partial \tilde{t}^{e}}\frac{%
\partial \tilde{t}^{f}}{\partial t^{b}}\frac{\partial x^{i}}{\partial \tilde{%
x}^{p}}\frac{\partial \tilde{x}^{q}}{\partial x^{j}}\left( \frac{\partial
x^{k}}{\partial \tilde{x}^{r}}\frac{\partial \tilde{t}^{g}}{\partial t^{c}}%
\right) \left( \frac{\partial t^{d}}{\partial \tilde{t}^{h}}\frac{\partial 
\tilde{x}^{s}}{\partial x^{l}}\right) ...
\end{equation*}%
is called a \textbf{distinguished tensor field} or a \textbf{d-tensor}.
\end{definition}

\begin{remark}
It is obvious that the vector fields $\delta /\delta t^{a},$ $\delta /\delta
x^{i},$ $\partial /\partial p_{i}^{a}$ of the adapted basis are $d$-vector
fields on $E^{\ast }$, as the adapted components $X^{a},$ $X^{i},$ $%
X_{\left( i\right) }^{\left( a\right) }$ of a vector field $X\in \chi \left(
E^{\ast }\right) $. Also, the covector fields $dt^{a},$ $dx^{i},$ $\delta
p_{i}^{a}$ of the adapted cobasis are $d$-covector fields on $E^{\ast }$, as
the adapted components $\omega _{a},$ $\omega _{i},$ $\omega _{\left(
a\right) }^{\left( i\right) }$ of an 1-form $\omega \in \chi ^{\ast }\left(
E^{\ast }\right) $. It follows that the set%
\begin{equation*}
\left\{ 1,\dfrac{\delta }{\delta t^{a}},\dfrac{\delta }{\delta x^{i}},\dfrac{%
\partial }{\partial p_{i}^{a}},dt^{a},dx^{i},\delta p_{i}^{a}\right\}
\end{equation*}%
generates the algebra of the $d$-tensor fields over the ring of functions $%
\mathcal{F}\left( E^{\ast }\right) $.
\end{remark}

\begin{example}
i) If $H:E^{\ast }\rightarrow \mathbb{R}$ is a Hamiltonian function
depending on the polymomenta $p_{i}^{a}$, then the local components%
\begin{equation*}
G_{\left( a\right) \left( b\right) }^{\left( i\right) \left( j\right) }=%
\frac{1}{2}\frac{\partial ^{2}H}{\partial p_{i}^{a}\partial p_{j}^{b}},
\end{equation*}%
represent a d-tensor field $\mathbb{G}=\left( G_{\left( a\right) \left(
b\right) }^{\left( i\right) \left( j\right) }\right) $. If $\mathcal{T}=%
\mathbb{R}$ and $H$ is a regular Hamiltonian function, then the d-tensor
field $\mathbb{G}$ can be regarded as the fundamental metrical $d$-tensor $%
g^{ij}\left( t,x,p\right) $ from the rheonomic Hamilton geometry.

ii) Let us consider the d-tensor field $\mathbb{C}^{\ast }=\left( \mathbb{C}%
_{\left( i\right) }^{\left( a\right) }\right) $, where $\mathbb{C}_{\left(
i\right) }^{\left( a\right) }=p_{i}^{a}.$ Particularly, for $\mathcal{T}=%
\mathbb{R},$ this d-tensor can be regarded as the classical
Liouville-Hamilton vector field%
\begin{equation*}
\mathbb{C}^{\ast }=p_{i}\dfrac{\partial }{\partial p_{i}}
\end{equation*}%
on the cotangent bundle $T^{\ast }M$, which is used in the Hamilton geometry
[17]. Consequently, our d-tensor field $\mathbb{C}^{\ast }$ on $E^{\ast }$
is called \textbf{the Liouville-Hamilton d-tensor field of polymomenta}.

iii) Let $\varphi _{ij}(x)$ be a semi-Riemannian metric on the spatial
manifold $M$. The geometrical object $\mathbb{H}=\left( H_{\left( i\right)
jk}^{\left( a\right) }\right) ,$ where $H_{\left( i\right) jk}^{\left(
a\right) }=\varphi _{ij}p_{k}^{a},$ is a d-tensor field on $E^{\ast }$,
which is called \textbf{the polymomentum Hamilton d-tensor attached to the
semi-Riemannian metric }$\varphi _{ij}(x)$.

iv) Let $h_{ab}(t)$ be a semi-Riemannian metric on the temporal manifold $%
\mathcal{T}.$ We can construct the $d$-tensor field $\mathbb{J}=\left(
J_{\left( a\right) bj}^{\left( i\right) }\right) ,$ where 
\begin{equation*}
J_{\left( a\right) bj}^{\left( i\right) }=h_{ab}\delta _{j}^{i}.
\end{equation*}%
The d-tensor field $\mathbb{J}$ is called the \textbf{d-tensor of }$h$%
\textbf{-normalization }on the dual $1$-jet vector bundle $E^{\ast }$.
\end{example}

\section{Poisson brackets. The almost product structure $\mathbb{P}$}

\hspace{4mm} In applications, the Poisson brackets of the d-vector fields $%
\left\{ \dfrac{\delta }{\delta t^{a}},\dfrac{\delta }{\delta x^{i}},\dfrac{%
\partial }{\partial p_{i}^{a}}\right\} $ from the adapted basis are very
important. By a direct calculation, we obtain

\begin{proposition}
The Poisson brackets of the $d$-vector fields of the adapted basis (\ref%
{ad-basis-nlc}) are given by%
\begin{equation}
\begin{array}{ll}
\left[ \dfrac{\delta }{\delta t^{b}},\dfrac{\delta }{\delta t^{c}}\right]
=R_{\left( i\right) bc}^{\left( a\right) }\dfrac{\partial }{\partial
p_{i}^{a}}, & \left[ \dfrac{\delta }{\delta t^{b}},\dfrac{\delta }{\delta
x^{k}}\right] =R_{\left( i\right) bk}^{\left( a\right) }\dfrac{\partial }{%
\partial p_{i}^{a}},\medskip \\ 
\left[ \dfrac{\delta }{\delta t^{b}},\dfrac{\partial }{\partial p_{k}^{c}}%
\right] =B_{\left( i\right) b\left( c\right) }^{\left( a\right) \ \left(
k\right) }\dfrac{\partial }{\partial p_{i}^{a}}, & \left[ \dfrac{\delta }{%
\delta x^{j}},\dfrac{\delta }{\delta x^{k}}\right] =R_{\left( i\right)
jk}^{\left( a\right) }\dfrac{\partial }{\partial p_{i}^{a}},\medskip \\ 
\left[ \dfrac{\delta }{\delta x^{j}},\dfrac{\partial }{\partial p_{k}^{c}}%
\right] =B_{\left( i\right) j\left( c\right) }^{\left( a\right) \ \left(
k\right) }\dfrac{\partial }{\partial p_{i}^{a}}, & \left[ \dfrac{\partial }{%
\partial p_{j}^{b}},\dfrac{\partial }{\partial p_{k}^{c}}\right] =0,%
\end{array}
\label{Poisson brackets}
\end{equation}%
where, if $\underset{1}{\overset{}{N}}\underset{}{\overset{}{_{\left(
i\right) b}^{(a)}}}$ and $\underset{2}{\overset{}{N}}\underset{}{\overset{%
_{{}}}{_{\left( i\right) j}^{(a)}}}$ are the local coefficients of the given
nonlinear connection $N$, then%
\begin{equation}
\begin{array}{l}
R_{\left( i\right) bc}^{\left( a\right) }=\dfrac{\delta \underset{1}{\overset%
{}{N}}\underset{}{\overset{}{_{\left( i\right) b}^{(a)}}}}{\delta t^{c}}-%
\dfrac{\delta \underset{1}{\overset{}{N}}\underset{}{\overset{}{_{\left(
i\right) c}^{(a)}}}}{\delta t^{b}},\medskip \\ 
R_{\left( i\right) bk}^{\left( a\right) }=\dfrac{\delta \underset{1}{\overset%
{}{N}}\underset{}{\overset{}{_{\left( i\right) b}^{(a)}}}}{\delta x^{k}}-%
\dfrac{\delta \underset{2}{\overset{}{N}}\underset{}{\overset{_{{}}}{%
_{\left( i\right) k}^{(a)}}}}{\delta t^{b}},\medskip \\ 
R_{\left( i\right) jk}^{\left( a\right) }=\dfrac{\delta \underset{2}{\overset%
{}{N}}\underset{}{\overset{_{{}}}{_{\left( i\right) j}^{(a)}}}}{\delta x^{k}}%
-\dfrac{\delta \underset{2}{\overset{}{N}}\underset{}{\overset{_{{}}}{%
_{\left( i\right) k}^{(a)}}}}{\delta x^{j}},\medskip \\ 
B_{\left( i\right) b\left( c\right) }^{\left( a\right) \ \left( k\right) }=%
\dfrac{\partial \underset{1}{\overset{}{N}}\underset{}{\overset{}{_{\left(
i\right) b}^{(a)}}}}{\partial p_{k}^{c}},\quad B_{\left( i\right) j\left(
c\right) }^{\left( a\right) \ \left( k\right) }=\dfrac{\partial \underset{2}{%
\overset{}{N}}\underset{}{\overset{_{{}}}{_{\left( i\right) j}^{(a)}}}}{%
\partial p_{k}^{c}}.%
\end{array}
\label{Formulas-Poisson-brackets}
\end{equation}
\end{proposition}

Using the relations (\ref{Poisson brackets}) and the Proposition \ref%
{H-integrability}, we get

\begin{proposition}
The horizontal distribution $\mathcal{H}$ is integrable if and only if the
following $d$-tensor fields vanish:%
\begin{equation}
R_{\left( i\right) bc}^{\left( a\right) }=0,\text{ }R_{\left( i\right)
bk}^{\left( a\right) }=0,\text{ }R_{\left( i\right) jk}^{\left( a\right) }=0.
\label{Cond-integrability-H}
\end{equation}
\end{proposition}

Now, assuming that a nonlinear connection $N$ is given, we can define an $%
\mathcal{F}\left( E^{\ast }\right) $-linear mapping $\mathbb{P}:\chi \left(
E^{\ast }\right) \rightarrow \chi \left( E^{\ast }\right) $, putting

\begin{equation}
\mathbb{P}\left( X^{\mathcal{H}_{\mathcal{T}}}\right) =X^{\mathcal{H}_{%
\mathcal{T}}},\text{ }\mathbb{P}\left( X^{\mathcal{H}_{M}}\right) =X^{%
\mathcal{H}_{M}},\text{ }\mathbb{P}\left( X^{\mathcal{W}}\right) =-X^{%
\mathcal{W}},\text{ }\forall \text{ }X\in \chi \left( E^{\ast }\right) .
\label{Defin-Apr-Prod}
\end{equation}%
Thus $\mathbb{P}$ has the properties:%
\begin{equation}
\begin{array}{ccc}
\mathbb{P}\circ \mathbb{P}=I, & \mathbb{P}=2\left( h_{\mathcal{T}%
}+h_{M}\right) -I=I-2w, & \text{rank }\mathbb{P}=m+n+mn.%
\end{array}
\label{Propr-Apr-Prod}
\end{equation}%
Obviously, we have

\begin{theorem}
A nonlinear connection $N$ on $E^{\ast }$ is characterized by the existence
of the \textbf{almost product structure} $\mathbb{P}$ on $E^{\ast }$, whose
eigenspace corresponding to the eingenvalue $-1$ coincides with the linear
space of the vertical distribution $\mathcal{W}$ on $E^{\ast }.$
\end{theorem}

Moreover, taking into account that the Nijenhuis tensor of the almost produs
structure $\mathbb{P}$ is given by%
\begin{equation*}
\mathcal{N}_{\mathbb{P}}\left( X,Y\right) =\mathbb{P}^{2}\left[ X,Y\right] +%
\left[ \mathbb{P}X,\mathbb{P}Y\right] -\mathbb{P}\left[ \mathbb{P}X,Y\right]
-\mathbb{P}\left[ X,\mathbb{P}Y\right] ,
\end{equation*}%
we obtain that%
\begin{equation*}
\begin{array}{ll}
\mathcal{N}_{\mathbb{P}}\left( X^{\mathcal{H}_{\mathcal{T}}},Y^{\mathcal{H}_{%
\mathcal{T}}}\right) =4w\left[ X^{\mathcal{H}_{\mathcal{T}}},Y^{\mathcal{H}_{%
\mathcal{T}}}\right] , & \mathcal{N}_{\mathbb{P}}\left( X^{\mathcal{H}_{%
\mathcal{T}}},Y^{\mathcal{H}_{M}}\right) =4w\left[ X^{\mathcal{H}_{\mathcal{T%
}}},Y^{\mathcal{H}_{M}}\right] ,\medskip \\ 
\mathcal{N}_{\mathbb{P}}\left( X^{\mathcal{H}_{\mathcal{T}}},Y^{\mathcal{W}%
}\right) =0, & \mathcal{N}_{\mathbb{P}}\left( X^{\mathcal{H}_{M}},Y^{%
\mathcal{H}_{M}}\right) =4w\left[ X^{\mathcal{H}_{M}},Y^{\mathcal{H}_{M}}%
\right] ,\medskip \\ 
\mathcal{N}_{\mathbb{P}}\left( X^{\mathcal{H}_{M}},Y^{\mathcal{W}}\right) =0,
& \mathcal{N}_{\mathbb{P}}\left( X^{\mathcal{W}},Y^{\mathcal{W}}\right) =0.%
\end{array}%
\end{equation*}

Therefore, we can formulate

\begin{proposition}
The almost product structure $\mathbb{P}$ is integrable if and only if the
horizontal distribution $\mathcal{H}$ is integrable.
\end{proposition}

\section{$N$-linear connections}

\hspace{4mm} A linear connection on $E^{\ast }=J^{1\ast }\left( \mathcal{T}%
,M\right) $ is an application%
\begin{equation*}
D:\chi \left( E^{\ast }\right) \times \chi \left( E^{\ast }\right)
\rightarrow \chi \left( E^{\ast }\right) ,\quad \left( X,Y\right)
\rightarrow D_{X}Y,
\end{equation*}%
having, for all $X,$ $X_{1},$ $X_{2},$ $Y_{1},$ $Y_{2},$ $Y\in \chi \left(
E^{\ast }\right) ,$ $f\in \mathcal{F}\left( E^{\ast }\right) $, the
properties:

\begin{enumerate}
\item $D_{X_{1}+X_{2}}Y=D_{X_{1}}Y+D_{X_{2}}Y,$

\item $D_{fX}Y=fD_{X}Y,$

\item $D_{X}\left( Y_{1}+Y_{2}\right) =D_{X}Y_{1}+D_{X}Y_{2},$

\item $D_{X}\left( fY\right) =X(f)Y+fD_{X}Y.$
\end{enumerate}

With respect to the decompositions of type (\ref{decomp-vect-fields}), we
have 
\begin{equation}
D_{X}Y=\underset{\alpha =0}{\overset{2}{\Sigma }}\left( D_{X^{\mathcal{H}_{%
\mathcal{T}}}}Y^{\mathcal{W}_{\alpha }}+D_{X^{\mathcal{H}_{M}}}Y^{\mathcal{W}%
_{\alpha }}+D_{X^{\mathcal{W}}}Y^{\mathcal{W}_{\alpha }}\right) ,
\label{Linear Connection}
\end{equation}%
where $\mathcal{W}_{1}=\mathcal{H}_{\mathcal{T}},$ $\mathcal{W}_{2}=\mathcal{%
H}_{M}$ and $\mathcal{W}_{0}=\mathcal{W}.$ Remark that the entities $D_{X^{%
\mathcal{H}_{\mathcal{T}}}}Y^{\mathcal{W}_{\alpha }},$ $D_{X^{\mathcal{H}%
_{M}}}Y^{\mathcal{W}_{\alpha }},$ $D_{X^{\mathcal{W}}}Y^{\mathcal{W}_{\alpha
}},$ $\alpha =1,2,0,$ are vector fields, not necessarily distinguished.

Obviously, the linear connection $D$ on $E^{\ast }$ can be uniquely
determined by \textbf{27} local coefficients, which are written in the
adapted basis (\ref{ad-basis-nlc}) in the form:%
\begin{eqnarray}
D_{\dfrac{\delta }{\delta t^{c}}}\dfrac{\delta }{\delta t^{b}} &=&A_{bc}^{a}%
\dfrac{\delta }{\delta t^{a}}+A_{bc}^{i}\dfrac{\delta }{\delta x^{i}}%
+A_{\left( i\right) bc}^{\left( a\right) }\dfrac{\partial }{\partial
p_{i}^{a}},  \label{D-temporal} \\
D_{\dfrac{\delta }{\delta t^{c}}}\dfrac{\delta }{\delta x^{j}} &=&A_{jc}^{a}%
\dfrac{\delta }{\delta t^{a}}+A_{jc}^{i}\dfrac{\delta }{\delta x^{i}}%
+A_{\left( i\right) jc}^{\left( a\right) }\dfrac{\partial }{\partial
p_{i}^{a}},  \notag \\
-D_{\dfrac{\delta }{\delta t^{c}}}\dfrac{\partial }{\partial p_{j}^{b}}
&=&A_{\ \left( b\right) c}^{a\left( j\right) }\dfrac{\delta }{\delta t^{a}}%
+A_{\ \left( b\right) c}^{i\left( j\right) }\dfrac{\delta }{\delta x^{i}}%
+A_{\left( i\right) \left( b\right) c}^{\left( a\right) \left( j\right) }%
\dfrac{\partial }{\partial p_{i}^{a}},  \notag
\end{eqnarray}

\begin{eqnarray}
D_{\dfrac{\delta }{\delta x^{k}}}\frac{\delta }{\delta t^{b}} &=&H_{bk}^{a}%
\frac{\delta }{\delta t^{a}}+H_{bk}^{i}\frac{\delta }{\delta x^{i}}%
+H_{\left( i\right) bk}^{\left( a\right) }\frac{\partial }{\partial p_{i}^{a}%
},  \label{D-spatial} \\
D_{\dfrac{\delta }{\delta x^{k}}}\frac{\delta }{\delta x^{j}} &=&H_{jk}^{a}%
\frac{\delta }{\delta t^{a}}+H_{jk}^{i}\frac{\delta }{\delta x^{i}}%
+H_{\left( i\right) jk}^{\left( a\right) }\frac{\partial }{\partial p_{i}^{a}%
},  \notag \\
-D_{\dfrac{\delta }{\delta x^{k}}}\frac{\partial }{\partial p_{j}^{b}}
&=&H_{\ \left( b\right) k}^{a\left( j\right) }\frac{\delta }{\delta t^{a}}%
+H_{\ \left( b\right) k}^{i\left( j\right) }\frac{\delta }{\delta x^{i}}%
+H_{\left( i\right) \left( b\right) k}^{\left( a\right) \left( j\right) }%
\frac{\partial }{\partial p_{i}^{a}},  \notag
\end{eqnarray}

\begin{eqnarray}
D_{\dfrac{\partial }{\partial p_{k}^{c}}}\frac{\delta }{\delta t^{b}}
&=&C_{b\left( c\right) }^{a\left( k\right) }\frac{\delta }{\delta t^{a}}%
+C_{b\left( c\right) }^{i\left( k\right) }\frac{\delta }{\delta x^{i}}%
+C_{\left( i\right) b\left( c\right) }^{\left( a\right) \ \left( k\right) }%
\frac{\partial }{\partial p_{i}^{a}},  \label{D-vertical} \\
D_{\dfrac{\partial }{\partial p_{k}^{c}}}\frac{\delta }{\delta x^{j}}
&=&C_{j\left( c\right) }^{a\left( k\right) }\frac{\delta }{\delta t^{a}}%
+C_{j\left( c\right) }^{i\left( k\right) }\frac{\delta }{\delta x^{i}}%
+C_{\left( i\right) j\left( c\right) }^{\left( a\right) \ \left( k\right) }%
\frac{\partial }{\partial p_{i}^{a}},  \notag \\
-D_{\dfrac{\partial }{\partial p_{k}^{c}}}\frac{\partial }{\partial p_{j}^{b}%
} &=&C_{\ \left( b\right) \left( c\right) }^{a\left( j\right) \left(
k\right) }\frac{\delta }{\delta t^{a}}+C_{\ \left( b\right) \left( c\right)
}^{i\left( j\right) \left( k\right) }\frac{\delta }{\delta x^{i}}+C_{\left(
i\right) \left( b\right) \left( c\right) }^{\left( a\right) \left( j\right)
\left( k\right) }\frac{\partial }{\partial p_{i}^{a}}.  \notag
\end{eqnarray}

To work with these 27 adapted coefficients is not imposible, but is
laborious. We construct in the sequel linear connections whose coefficients
are much easy to be used. In this direction, let us consider a nonlinear
connection $N$ on $E^{\ast }.$

\begin{definition}
A linear connection $D$ on $E^{\ast }$ is called an $N$\textbf{-linear
connection} if it preserves by parallelism the $\mathcal{T}$-horizontal, $M$%
-horizontal and vertical distributions $\mathcal{H}_{\mathcal{T}},$ $%
\mathcal{H}_{M}$ and $\mathcal{W}$ on $E^{\ast }.$
\end{definition}

In other words, a linear connection is an $N$-linear connection if and only
if, for any vector field $X\in \chi \left( E^{\ast }\right) ,$ $D_{X}$
carries the $\mathcal{T}$-horizontal vector fields into $\mathcal{T}$%
-horizontal vector fields, the $M$-horizontal vector fields into $M$%
-horizontal vector fields and the vertical vector fields into vertical
vector fields. Thus, we have $D_{X}Y^{\mathcal{H}_{\mathcal{T}}}\in \chi
\left( \mathcal{H}_{\mathcal{T}}\right) ,$ $D_{X}Y^{\mathcal{H}_{M}}\in \chi
\left( \mathcal{H}_{M}\right) $ and $D_{X}Y^{\mathcal{W}}\in \chi \left( 
\mathcal{W}\right) ,$ which can be written in the form%
\begin{equation}
D_{X}\left( h_{\mathcal{T}}Y\right) =h_{\mathcal{T}}D_{X}Y,\text{ }%
D_{X}\left( h_{M}Y\right) =h_{M}D_{X}Y,\text{ }D_{X}\left( wY\right)
=wD_{X}Y.  \label{D-preserves distrib}
\end{equation}

Consequently, we obtain

\begin{theorem}
A linear connection $D$ is an $N$-linear connection if and only if, for any
vector field $X\in \chi \left( E^{\ast }\right) $, we have%
\begin{equation}
D_{X}h_{\mathcal{T}}=0,\ D_{X}h_{M}=0,\ D_{X}w=0.  \label{D-h-uri=0}
\end{equation}
\end{theorem}

\begin{corollary}
For any $N$-linear connection $D$ we have%
\begin{equation}
D_{X}\mathbb{P}=0,\text{ }\forall \text{ }X\in \chi \left( E^{\ast }\right) .
\label{DP=0}
\end{equation}
\end{corollary}

Moreover, we also have

\begin{theorem}
A linear connection $D$ on $E^{\ast }$ is an $N$-linear connection if and
only if%
\begin{equation*}
\begin{array}{ll}
\left( D_{X}Y^{\mathcal{W}_{\beta }}\right) ^{\mathcal{W}}=0, & \left(
D_{X}Y^{\mathcal{W}}\right) ^{\mathcal{W}_{\beta }}=0,\medskip  \\ 
\left( D_{X}Y^{\mathcal{W}_{1}}\right) ^{\mathcal{W}_{2}}=0, & \left(
D_{X}Y^{\mathcal{W}_{2}}\right) ^{\mathcal{W}_{1}}=0,%
\end{array}%
\end{equation*}%
where $\beta =1,2$ and $\mathcal{W}_{1}=\mathcal{H}_{\mathcal{T}},$ $%
\mathcal{W}_{2}=\mathcal{H}_{M}$.
\end{theorem}

Consequently, using the adapted basis of vector fields on $E^{\ast }$, given
by (\ref{ad-basis-nlc}), and the above results, we prove without difficulties

\begin{theorem}
An $N$-linear connection can be uniquely written in the adapted basis of
vector fields on $E^{\ast \text{ }}$with \textbf{9} adapted coefficients
given by the relations:%
\begin{eqnarray*}
D_{\dfrac{\delta }{\delta t^{c}}}\frac{\delta }{\delta t^{b}} &=&A_{bc}^{a}%
\frac{\delta }{\delta t^{a}},\text{ }D_{\dfrac{\delta }{\delta t^{c}}}\frac{%
\delta }{\delta x^{j}}=A_{jc}^{i}\frac{\delta }{\delta x^{i}},\text{ }D_{%
\dfrac{\delta }{\delta t^{c}}}\frac{\partial }{\partial p_{j}^{b}}%
=-A_{\left( i\right) \left( b\right) c}^{\left( a\right) \left( j\right) }%
\frac{\partial }{\partial p_{i}^{a}}, \\
D_{\dfrac{\delta }{\delta x^{k}}}\frac{\delta }{\delta t^{b}} &=&H_{bk}^{a}%
\frac{\delta }{\delta t^{a}},\text{ }D_{\dfrac{\delta }{\delta x^{k}}}\frac{%
\delta }{\delta x^{j}}=H_{jk}^{i}\frac{\delta }{\delta x^{i}},\text{ }D_{%
\dfrac{\delta }{\delta x^{k}}}\frac{\partial }{\partial p_{j}^{b}}%
=-H_{\left( i\right) \left( b\right) k}^{\left( a\right) \left( j\right) }%
\frac{\partial }{\partial p_{i}^{a}}, \\
D_{\dfrac{\partial }{\partial p_{k}^{c}}}\frac{\delta }{\delta t^{b}}
&=&C_{b\left( c\right) }^{a\left( k\right) }\frac{\delta }{\delta t^{a}},%
\text{ }D_{\dfrac{\partial }{\partial p_{k}^{c}}}\frac{\delta }{\delta x^{j}}%
=C_{j\left( c\right) }^{i\left( k\right) }\frac{\delta }{\delta x^{i}},\text{
}D_{\dfrac{\partial }{\partial p_{k}^{c}}}\frac{\partial }{\partial p_{j}^{b}%
}=-C_{\left( i\right) \left( b\right) \left( c\right) }^{\left( a\right)
\left( j\right) \left( k\right) }\frac{\partial }{\partial p_{i}^{a}}.
\end{eqnarray*}
\end{theorem}

\begin{definition}
The local functions%
\begin{equation}
D\Gamma \left( N\right) =\left( A_{bc}^{a},A_{jc}^{i},A_{\left( i\right)
\left( b\right) c}^{\left( a\right) \left( j\right)
},H_{bk}^{a},H_{jk}^{i},H_{\left( i\right) \left( b\right) k}^{\left(
a\right) \left( j\right) },C_{b\left( c\right) }^{a\left( k\right)
},C_{j\left( c\right) }^{i\left( k\right) },C_{\left( i\right) \left(
b\right) \left( c\right) }^{\left( a\right) \left( j\right) \left( k\right)
}\right)  \label{Local-descr-D-Nine-Coeff}
\end{equation}%
are called the \textbf{adapted coefficients} of the $N$-linear connection $D$
on $E^{\ast }.$
\end{definition}

Taking into acount the transformation laws (\ref{tr-laws-ad-basis}) of the
d-vector fields of the adapted basis (\ref{ad-basis-nlc}), by a
straightforward calculation, we obtain

\begin{theorem}
i) With respect to the coordinate transformations (\ref{tr-group-dual-jet})
on $E^{\ast }$, the adapted coefficients of the $N$-linear connection $%
D\Gamma \left( N\right) $ obey the following rules of
transformations:\bigskip

$\left( h_{\mathcal{T}\text{ }}\right) $ $\left\{ 
\begin{array}{l}
A_{bc}^{d}=\tilde{A}_{fg}^{a}\dfrac{\partial \tilde{t}^{f}}{\partial t^{b}}%
\dfrac{\partial \tilde{t}^{g}}{\partial t^{c}}\dfrac{\partial t^{d}}{%
\partial \tilde{t}^a}+\dfrac{\partial t^{d}}{\partial \tilde{t}^a}\dfrac{%
\partial ^{2}\tilde{t}^{a}}{\partial t^{b}\partial t^{c}},\medskip \\ 
A_{jc}^{i}=\tilde{A}_{ld}^{k}\dfrac{\partial x^{i}}{\partial \tilde{x}^{k}}%
\dfrac{\partial \tilde{x}^{l}}{\partial x^{j}}\dfrac{\partial \tilde{t}^{d}}{%
\partial t^{c}},\medskip \\ 
A_{\left( i\right) \left( b\right) c}^{\left( a\right) \left( j\right) }=%
\tilde{A}_{\left( k\right) \left( f\right) g}^{\left( d\right) \left(
l\right) }\dfrac{\partial t^{a}}{\partial \tilde{t}^{d}}\dfrac{\partial 
\tilde{x}^{k}}{\partial x^{i}}\dfrac{\partial x^{j}}{\partial \tilde{x}^{l}}%
\dfrac{\partial \tilde{t}^{f}}{\partial t^{b}}\dfrac{\partial \tilde{t}^{g}}{%
\partial t^{c}}-\delta _{i}^{j}\dfrac{\partial t^{a}}{\partial \tilde{t}^{d}}%
\dfrac{\partial ^{2}\tilde{t}^{d}}{\partial t^{b}\partial t^{c}},%
\end{array}%
\right. \bigskip$

$\left( h_{M}\right) $ $\left\{ 
\begin{array}{l}
H_{bk}^{a}=\tilde{H}_{dl}^{c}\dfrac{\partial t^{a}}{\partial \tilde{t}^{c}}%
\dfrac{\partial \tilde{t}^{d}}{\partial t^{b}}\dfrac{\partial \tilde{x}^{l}}{%
\partial x^{k}},\medskip \\ 
H_{jk}^{l}=\tilde{H}_{rs}^{i}\dfrac{\partial \tilde{x}^{r}}{\partial x^{j}}%
\dfrac{\partial \tilde{x}^{s}}{\partial x^{k}}\dfrac{\partial x^{l}}{%
\partial \tilde{x}^i}+\dfrac{\partial x^{l}}{\partial \tilde{x}^i}\dfrac{%
\partial ^{2}\tilde{x}^{i}}{\partial x^{j}\partial x^{k}},\medskip \\ 
H_{\left( i\right) \left( b\right) k}^{\left( a\right) \left( j\right) }=%
\tilde{H}_{\left( r\right) \left( g\right) s}^{\left( f\right) \left(
l\right) }\dfrac{\partial t^{a}}{\partial \tilde{t}^{f}}\dfrac{\partial 
\tilde{x}^{r}}{\partial x^{i}}\dfrac{\partial x^{j}}{\partial \tilde{x}^{l}}%
\dfrac{\partial \tilde{t}^{g}}{\partial t^{b}}\dfrac{\partial \tilde{x}^{s}}{%
\partial x^{k}}-\delta _{b}^{a}\dfrac{\partial \tilde{x}^{r}}{\partial x^{i}}%
\dfrac{\partial \tilde{x}^{s}}{\partial x^{k}}\dfrac{\partial ^{2}x^{j}}{%
\partial \tilde{x}^{r}\partial \tilde{x}^{s}},%
\end{array}%
\right. \bigskip$

$\left( \text{ }w_{\text{ }}\right) $ $\left\{ 
\begin{array}{l}
C_{b\left( c\right) }^{a\left( k\right) }=\tilde{C}_{f\left( g\right)
}^{d\left( r\right) }\dfrac{\partial t^{a}}{\partial \tilde{t}^{d}}\dfrac{%
\partial \tilde{t}^{f}}{\partial t^{b}}\dfrac{\partial x^{k}}{\partial 
\tilde{x}^{r}}\dfrac{\partial \tilde{t}^{g}}{\partial t^{c}},\medskip \\ 
C_{j\left( c\right) }^{i\left( k\right) }=\tilde{C}_{l\left( f\right)
}^{r\left( s\right) }\dfrac{\partial x^{i}}{\partial \tilde{x}^{r}}\dfrac{%
\partial \tilde{x}^{l}}{\partial x^{j}}\dfrac{\partial \tilde{t}^{f}}{%
\partial t^{c}}\dfrac{\partial x^{k}}{\partial \tilde{x}^{s}},\medskip \\ 
C_{\left( i\right) \left( b\right) \left( c\right) }^{\left( a\right) \left(
j\right) \left( k\right) }=\tilde{C}_{\left( l\right) \left( f\right) \left(
g\right) }^{\left( d\right) \left( r\right) \left( s\right) }\dfrac{\partial
t^{a}}{\partial \tilde{t}^{d}}\dfrac{\partial \tilde{x}^{l}}{\partial x^{i}}%
\dfrac{\partial x^{j}}{\partial \tilde{x}^{r}}\dfrac{\partial \tilde{t}^{f}}{%
\partial t^{b}}\dfrac{\partial x^{k}}{\partial \tilde{x}^{s}}\dfrac{\partial 
\tilde{t}^{g}}{\partial t^{c}}.%
\end{array}%
\right. \bigskip$

ii) Conversely, to give an $N$-linear connection $D$ on $E^{\ast }$ is
equivalent to give a set of nine local coefficients $D\Gamma \left( N\right) 
$ as in (\ref{Local-descr-D-Nine-Coeff}), which transform by the rules
described in i).
\end{theorem}

The following result proves the existence of the $N$-linear connections on $%
E^{\ast }.$

\begin{theorem}
If the manifolds $\mathcal{T}$ and $M$ are paracompacts, then there exists
an $N$-linear connection on $E^{\ast }.$
\end{theorem}

\begin{proof}
Because the manifold $\mathcal{T}$ (resp. $M$) is paracompact, there exists
a linear connection on $\mathcal{T}$ (resp. $M$), whose local coefficients
we denote by $\chi _{bc}^{a}\left( t\right) $ (resp. $\Gamma _{jk}^{i}\left(
x\right) $). Let us consider the local coefficients%
\begin{equation}
\overset{B}{\underset{1}{N}}\underset{}{\overset{_{{}}}{_{\left( i\right)
b}^{(a)}}}=\chi _{cb}^{a}p_{i}^{c},\ \overset{B}{\underset{2}{N}}\underset{}{%
\overset{_{{}}}{_{\left( j\right) k}^{(b)}}}=-\Gamma _{jk}^{i}p_{i}^{b},
\label{nlc-assoc-to linear connections}
\end{equation}%
which define a nonlinear connection $\overset{B}{N}$ on $E^{\ast }$. We set%
\begin{equation}
B\Gamma \left( \overset{B}{N}\right) =\left( \chi _{bc}^{a},0,A_{\left(
i\right) \left( b\right) c}^{\left( a\right) \left( j\right) },0,\Gamma
_{jk}^{i},H_{\left( i\right) \left( b\right) k}^{\left( a\right) \left(
j\right) },0,0,0\right)  \label{Berwald-connection}
\end{equation}%
where%
\begin{equation}
A_{\left( i\right) \left( b\right) c}^{\left( a\right) \left( j\right)
}=-\delta _{i}^{j}\chi _{bc}^{a},\ H_{\left( i\right) \left( b\right)
k}^{\left( a\right) \left( j\right) }=\delta _{b}^{a}\Gamma _{ik}^{j}.
\label{Def-Berwald-Connection}
\end{equation}%
Then, $B\Gamma \left( \overset{B}{N}\right) $ defines an $N$-linear
connection on $E^{\ast }.$
\end{proof}

\begin{definition}
i) The $N$-linear connection $B\Gamma \left( \overset{B}{N}\right) $ on $%
E^{\ast }$, which is given by the relations (\ref{nlc-assoc-to linear
connections}), (\ref{Berwald-connection}) and (\ref{Def-Berwald-Connection}%
), is called the \textbf{canonical Berwald connection attached to the linear
connections} $\chi _{bc}^{a}\left( t\right) $ and $\Gamma _{jk}^{i}\left(
x\right) $.

ii) If we have $\overset{B}{N}=\overset{0}{N}$, where $\overset{0}{N}$ is
given by (\ref{can-nlc-asoc-to-metric}), then $B\Gamma \left( \overset{0}{N}%
\right) $ is called the \textbf{canonical Berwald connection attached to the
semi-Riemannian metrics} $h_{ab}(t)$ \textbf{and }$\varphi _{ij}(x).$
\end{definition}

Now, let us consider that $D$ is a fixed $N$-linear connection on $E^{\ast }$%
, defined by the adapted local coefficients (\ref{Local-descr-D-Nine-Coeff}%
). The linear connection $D\Gamma \left( N\right) $ naturally induces
derivations on the set of the d-tensor fields on the dual 1-jet bundle $%
E^{\ast }$, in the following way. Starting from a d-vector field $X\in \chi
\left( E^{\ast }\right) $ and a d-tensor field $T$ on \thinspace $E^{\ast }$%
, locally expressed by%
\begin{eqnarray*}
X &=&X^{a}\frac{\delta }{\delta t^{a}}+X^{i}\frac{\delta }{\delta x^{i}}%
+X_{\left( i\right) }^{\left( a\right) }\frac{\partial }{\partial p_{i}^{a}},
\\
T &=&T_{cj\left( b\right) \left( l\right) ...}^{ai\left( k\right) \left(
d\right) ...}\frac{\delta }{\delta t^{a}}\otimes \frac{\delta }{\delta x^{i}}%
\otimes \frac{\partial }{\partial p_{l}^{d}}\otimes dt^{c}\otimes
dx^{j}\otimes \delta p_{k}^{b}\otimes ...,
\end{eqnarray*}%
we obtain

\begin{equation*}
\begin{array}{l}
D_{X}T=X^{g}D_{\dfrac{\delta }{\delta t^{g}}}T+X^{s}D_{\dfrac{\delta }{%
\delta x^{s}}}T+X_{\left( s\right) }^{\left( g\right) }D_{\dfrac{\partial }{%
\partial p_{s}^{g}}}T=\medskip \\ 
=\left\{ X^{g}T_{cj\left( b\right) \left( l\right) .../g}^{ai\left( k\right)
\left( d\right) ...}+X^{s}T_{cj\left( b\right) \left( l\right)
...|s}^{ai\left( k\right) \left( d\right) ...}+\right. \medskip \\ 
\left. +X_{\left( s\right) }^{\left( g\right) }T_{cj\left( b\right) \left(
l\right) ...}^{ai\left( k\right) \left( d\right) ...}\mid _{\left( g\right)
}^{\left( s\right) }\right\} \dfrac{\delta }{\delta t^{a}}\otimes \dfrac{%
\delta }{\delta x^{i}}\otimes \dfrac{\partial }{\partial p_{l}^{d}}\otimes
dt^{c}\otimes dx^{j}\otimes \delta p_{k}^{b}\otimes ...,%
\end{array}%
\end{equation*}%
where\medskip

$\left( h_{\mathcal{T}}\text{ }\right) $ $\left\{ 
\begin{array}{l}
T_{cj\left( b\right) \left( l\right) .../g}^{ai\left( k\right) \left(
d\right) ...}=\dfrac{\delta T_{cj\left( b\right) \left( l\right)
...}^{ai\left( k\right) \left( d\right) ...}}{\delta t^{g}}+T_{cj\left(
b\right) \left( l\right) ...}^{fi\left( k\right) \left( d\right)
...}A_{fg}^{a}+\medskip \\ 
+T_{cj\left( b\right) \left( l\right) ...}^{ar\left( k\right) \left(
d\right) ...}A_{rg}^{i}+T_{cj\left( f\right) \left( l\right) ...}^{ai\left(
r\right) \left( d\right) ...}A_{\left( r\right) \left( b\right) g}^{\left(
f\right) \left( k\right) }+...-\medskip \\ 
-T_{fj\left( b\right) \left( l\right) ...}^{ai\left( k\right) \left(
d\right) ...}A_{cg}^{f}-T_{cr\left( b\right) \left( l\right) ...}^{ai\left(
k\right) \left( d\right) ...}A_{jg}^{r}-T_{cj\left( b\right) \left( r\right)
...}^{ai\left( k\right) \left( f\right) ...}A_{\left( l\right) \left(
f\right) g}^{\left( d\right) \left( r\right) }-...,%
\end{array}%
\right. \medskip$

$\left( h_{M}\right) $ $\left\{ 
\begin{array}{l}
T_{cj\left( b\right) \left( l\right) ...|s}^{ai\left( k\right) \left(
d\right) ...}=\dfrac{\delta T_{cj\left( b\right) \left( l\right)
...}^{ai\left( k\right) \left( d\right) ...}}{\delta x^{s}}+T_{cj\left(
b\right) \left( l\right) ...}^{fi\left( k\right) \left( d\right)
...}H_{fs}^{a}+\medskip \\ 
+T_{cj\left( b\right) \left( l\right) ...}^{ar\left( k\right) \left(
d\right) ...}H_{rs}^{i}+T_{cj\left( f\right) \left( l\right) ...}^{ai\left(
r\right) \left( d\right) ...}H_{\left( r\right) \left( b\right) s}^{\left(
f\right) \left( k\right) }+...-\medskip \\ 
-T_{fj\left( b\right) \left( l\right) ...}^{ai\left( k\right) \left(
d\right) ...}H_{cs}^{f}-T_{cr\left( b\right) \left( l\right) ...}^{ai\left(
k\right) \left( d\right) ...}H_{js}^{r}-T_{cj\left( b\right) \left( r\right)
...}^{ai\left( k\right) \left( f\right) ...}H_{\left( l\right) \left(
f\right) s}^{\left( d\right) \left( r\right) }-...,%
\end{array}%
\right. \medskip $

$\left( \text{ }w_{\text{ }}\right) $ $\left\{ 
\begin{array}{l}
T_{cj\left( b\right) \left( l\right) ...}^{ai\left( k\right) \left( d\right)
...}\left\vert _{\left( g\right) }^{\left( s\right) }\right. =\dfrac{%
\partial T_{cj\left( b\right) \left( l\right) ...}^{ai\left( k\right) \left(
d\right) ...}}{\partial p_{s}^{g}}+T_{cj\left( b\right) \left( l\right)
...}^{fi\left( k\right) \left( d\right) ...}C_{f\left( g\right) }^{a\left(
s\right) }+\medskip \\ 
+T_{cj\left( b\right) \left( l\right) ...}^{ar\left( k\right) \left(
d\right) ...}C_{r\left( g\right) }^{i\left( s\right) }+T_{cj\left( f\right)
\left( l\right) ...}^{ai\left( r\right) \left( d\right) ...}C_{\left(
r\right) \left( b\right) \left( g\right) }^{\left( f\right) \left( k\right)
\left( s\right) }+...-\medskip \\ 
-T_{fj\left( b\right) \left( l\right) ...}^{ai\left( k\right) \left(
d\right) ...}C_{c\left( g\right) }^{f\left( s\right) }-T_{cr\left( b\right)
\left( l\right) ...}^{ai\left( k\right) \left( d\right) ...}C_{j\left(
g\right) }^{r\left( s\right) }-T_{cj\left( b\right) \left( r\right)
...}^{ai\left( k\right) \left( f\right) ...}C_{\left( l\right) \left(
f\right) \left( g\right) }^{\left( d\right) \left( r\right) \left( s\right)
}-...\text{.}%
\end{array}%
\right. $

\begin{definition}
The local derivative operators $"_{/a}",\ "_{|i}"$ and $"\mid _{\left(
a\right) }^{\left( i\right) }"$ are called the \textbf{$\mathcal{T}$%
-horizontal covariant derivative,} the $M$\textbf{-horizontal covariant
derivative} and the \textbf{vertical covariant derivative} \textbf{attached
to the }$N$\textbf{-linear connection }$D\Gamma \left( N\right) $. They are
applied to the local components of an arbitrary d-tensor $T$ on the dual
1-jet spaces $E^{\ast }$.
\end{definition}

By a direct calculation, we obtain

\begin{proposition}
The operators $"_{/a}",$ $"_{|i}"$ and $"\mid _{\left( a\right) }^{\left(
i\right) }"$ have the properties:

i) They are distributive with respect to the addition of the d-tensor fields
of the same type.

ii) They commute with the operation of contraction.

iii) They verify the Leibniz rule with respect to the tensor product.
\end{proposition}

\begin{remark}
i) If $T=f$ is a function on $E^{\ast }$, then the following expressions of
the local covariant derivatives hold good:%
\begin{equation*}
f_{/b}=\dfrac{\delta f}{\delta t^{b}}=\dfrac{\partial f}{\partial t^{b}}-%
\underset{1}{\overset{}{N}}\underset{}{\overset{}{_{\left( i\right) b}^{(a)}}%
}\dfrac{\partial f}{\partial p_{i}^{a}},\quad f_{|j}=\dfrac{\delta f}{\delta
x^{j}}=\dfrac{\partial f}{\partial x^{j}}-\underset{\left( 2\right) }{%
\overset{}{N}}\underset{}{\overset{_{{}}}{_{\left( i\right) j}^{(a)}}}\dfrac{%
\partial f}{\partial p_{i}^{a}},\quad f|_{\left( a\right) }^{\left( i\right)
}=\dfrac{\partial f}{\partial p_{i}^{a}}.
\end{equation*}

ii) If $T=Y$ is a d-vector field on $E^{\ast }$, locally expressed by%
\begin{equation*}
Y=Y^{a}\frac{\delta }{\delta t^{a}}+Y^{i}\frac{\delta }{\delta x^{i}}%
+Y_{\left( i\right) }^{\left( a\right) }\frac{\partial }{\partial p_{i}^{a}},
\end{equation*}%
then the following expressions of the local covariant derivatives hold good:%
\begin{equation*}
\left( h_{\mathcal{T}}\right) \left\{ 
\begin{array}{l}
Y_{{}}^{a}\overset{}{_{/c}}=\dfrac{\delta Y^{a}}{\delta t^{c}}%
+Y^{b}A_{bc}^{a},\medskip \\ 
Y_{{}}^{i}\overset{}{_{/c}}=\dfrac{\delta Y^{i}}{\delta t^{c}}%
+Y^{j}A_{jc}^{i},\medskip \\ 
Y_{\left( i\right) /c}^{\left( a\right) }=\dfrac{\delta Y_{\left( i\right)
}^{\left( a\right) }}{\delta t^{c}}-Y_{\left( j\right) }^{\left( b\right)
}A_{\left( i\right) \left( b\right) c}^{\left( a\right) \left( j\right) },%
\end{array}%
\right. \text{ }\left( h_{M}\right) \left\{ 
\begin{array}{l}
Y_{{}}^{a}\overset{}{_{|k}}=\dfrac{\delta Y^{a}}{\delta x^{k}}%
+Y^{b}H_{bk}^{a},\medskip \\ 
Y_{{}}^{i}\overset{}{_{|k}}=\dfrac{\delta Y^{i}}{\delta x^{k}}%
+Y^{j}H_{jk}^{i},\medskip \\ 
Y_{\left( i\right) |k}^{\left( a\right) }=\dfrac{\delta Y_{\left( i\right)
}^{\left( a\right) }}{\delta x^{k}}-Y_{\left( j\right) }^{\left( b\right)
}H_{\left( i\right) \left( b\right) k}^{\left( a\right) \left( j\right) },%
\end{array}%
\right.
\end{equation*}%
\begin{equation*}
\left( w\right) \left\{ 
\begin{array}{l}
Y_{{}}^{i}\mid _{\left( c\right) }^{\left( k\right) }=\dfrac{\partial Y^{a}}{%
\partial p_{k}^{c}}+Y^{b}C_{b\left( c\right) }^{a\left( k\right) },\medskip
\\ 
Y_{{}}^{i}\mid _{\left( c\right) }^{\left( k\right) }=\dfrac{\partial Y^{i}}{%
\partial p_{k}^{c}}+Y^{j}C_{j\left( c\right) }^{i\left( k\right) },\medskip
\\ 
Y_{\left( i\right) }^{\left( a\right) }\mid _{\left( c\right) }^{\left(
k\right) }=\dfrac{\partial Y_{\left( i\right) }^{\left( a\right) }}{\partial
p_{k}^{c}}-Y_{\left( j\right) }^{\left( b\right) }C_{\left( i\right) \left(
b\right) \left( c\right) }^{\left( a\right) \left( j\right) \left( k\right)
}.%
\end{array}%
\right.
\end{equation*}%
iii) If $T=\omega $ is a d-covector field on $E^{\ast }$, locally expressed
by%
\begin{equation*}
\omega =\omega _{a}dx^{a}+\omega _{i}dx^{i}+\omega _{\left( a\right)
}^{\left( i\right) }\delta p_{i}^{a},
\end{equation*}%
then the following expressions of the local covariant derivatives hold good:%
\begin{equation*}
\left( h_{\mathcal{T}}\right) \left\{ 
\begin{array}{l}
\omega _{a/c}=\dfrac{\delta \omega _{a}}{\delta t^{c}}-A_{ac}^{b}\omega
_{b},\medskip \\ 
\omega _{i/c}=\dfrac{\delta \omega _{i}}{\delta t^{c}}-A_{ic}^{j}\omega
_{j},\medskip \\ 
\omega _{\left( a\right) /c}^{\left( i\right) }=\dfrac{\delta \omega
_{\left( a\right) }^{\left( i\right) }}{\delta t^{c}}+A_{\left( j\right)
\left( a\right) c}^{\left( b\right) \left( i\right) }\omega _{\left(
b\right) }^{\left( j\right) },%
\end{array}%
\right. \text{ }\left( h_{M}\right) \left\{ 
\begin{array}{l}
\omega _{a|k}=\dfrac{\delta \omega _{a}}{\delta x^{k}}-H_{ak}^{b}\omega
_{b},\medskip \\ 
\omega _{i|k}=\dfrac{\delta \omega _{i}}{\delta x^{k}}-H_{ik}^{j}\omega
_{j},\medskip \\ 
\omega _{\left( a\right) |k}^{\left( i\right) }=\dfrac{\delta \omega
_{\left( a\right) }^{\left( i\right) }}{\delta x^{k}}+H_{\left( j\right)
\left( a\right) k}^{\left( b\right) \left( i\right) }\omega _{\left(
b\right) }^{\left( j\right) },%
\end{array}%
\right.
\end{equation*}%
\begin{equation*}
\left( w\right) \left\{ 
\begin{array}{l}
\omega _{a}\mid _{\left( c\right) }^{\left( k\right) }=\dfrac{\partial
\omega _{a}}{\partial p_{k}^{c}}-C_{a\left( c\right) }^{b\left( k\right)
}\omega _{b},\medskip \\ 
\omega _{i}\mid _{\left( c\right) }^{\left( k\right) }=\dfrac{\partial
\omega _{i}}{\partial p_{k}^{c}}-C_{i\left( c\right) }^{j\left( k\right)
}\omega _{j},\medskip \\ 
\omega _{\left( a\right) }^{\left( i\right) }\mid _{\left( c\right)
}^{\left( k\right) }=\dfrac{\partial \omega _{\left( a\right) }^{\left(
i\right) }}{\partial p_{k}^{c}}+C_{\left( j\right) \left( a\right) \left(
c\right) }^{\left( b\right) \left( i\right) \left( k\right) }\omega
_{(b)}^{(j)}.%
\end{array}%
\right.
\end{equation*}%
iv)\ Taking into account that for any $1$-form $\omega \in \chi ^{\ast
}\left( E^{\ast }\right) $ we have, by definition, 
\begin{equation*}
\left( D_{X}\omega \right) \left( Y\right) =X\omega \left( Y\right) -\omega
\left( D_{X}Y\right) ,\text{ }\forall \text{ }X,Y\in \chi \left( E^{\ast
}\right) ,
\end{equation*}%
we find the rules of covariant derivatives for adapted cobasis (\ref%
{ad-cobasis-nlc}) of covector fields on $E^{\ast }$, as following:%
\begin{equation*}
D_{\dfrac{\delta }{\delta t^{c}}}dt^{a}=-A_{bc}^{a}dt^{b},\quad D_{\dfrac{%
\delta }{\delta t^{c}}}dx^{i}=-A_{jc}^{i}dx^{j},\quad D_{\dfrac{\delta }{%
\delta t^{c}}}\delta p_{i}^{a}=A_{\left( i\right) \left( b\right) c}^{\left(
a\right) \left( j\right) }\delta p_{j}^{b},
\end{equation*}%
\begin{equation*}
D_{\dfrac{\delta }{\delta x^{k}}}dt^{a}=-H_{bk}^{a}dt^{b},\quad D_{\dfrac{%
\delta }{\delta x^{k}}}dx^{i}=-H_{jk}^{i}dx^{j},\quad D_{\dfrac{\delta }{%
\delta x^{k}}}\delta p_{i}^{a}=H_{\left( i\right) \left( b\right) k}^{\left(
a\right) \left( j\right) }\delta p_{j}^{b},
\end{equation*}%
\begin{equation*}
D_{\dfrac{\partial }{\partial p_{k}^{c}}}dt^{a}=-C_{b\left( c\right)
}^{a\left( k\right) }dt^{b},\text{ }D_{\dfrac{\partial }{\partial p_{k}^{c}}%
}dx^{i}=-C_{j\left( c\right) }^{i\left( k\right) }dx^{j},\text{ }D_{\dfrac{%
\partial }{\partial p_{k}^{c}}}\delta p_{i}^{a}=C_{\left( i\right) \left(
b\right) \left( c\right) }^{\left( a\right) \left( j\right) \left( k\right)
}\delta p_{j}^{b}.
\end{equation*}%
v) In the particular case of the canonical Berwald $\overset{B}{N}$-linear
connection $B\Gamma \left( \overset{B}{N}\right) $, defined by the relations
(\ref{nlc-assoc-to linear connections}), (\ref{Berwald-connection}) and (\ref%
{Def-Berwald-Connection}), the local covariant derivatives are denoted by $%
"_{//a}","_{||i}"$ and $"\parallel _{\left( a\right) }^{\left( i\right) }".$
\end{remark}

Now, we shall give an application of this paragraph. In this direction, let
us consider the \textit{canonical Liouville-Hamilton d-tensor field of
polymomenta} on $E^{\ast }$, given by%
\begin{equation*}
\mathbb{C}^{\ast }=\mathbb{C}_{(i)}^{(a)}\frac{\partial }{\partial p_{i}^{a}}%
,
\end{equation*}%
where $\mathbb{C}_{(i)}^{(a)}=p_{i}^{a}.$

\begin{definition}
The d-tensor fields defined by 
\begin{equation}
\Delta _{\left( i\right) b}^{\left( a\right) }=\mathbb{C}_{(i)/b}^{(a)},%
\quad \Delta _{\left( i\right) j}^{\left( a\right) }=\mathbb{C}%
_{(i)|j}^{(a)},\quad \vartheta _{\left( i\right) \left( b\right) }^{\left(
a\right) \left( j\right) }=\mathbb{C}_{(i)}^{(a)}|_{\left( b\right)
}^{\left( j\right) },  \label{Def-Defl-tensors}
\end{equation}%
are called the \textbf{polymomentum deflection d-tensor fields attached to
the }$N$\textbf{-linear connection }$D\Gamma \left( N\right) $.
\end{definition}

By a direct calculation, we find

\begin{proposition}
The polymomentum deflection d-tensor fields on $E^{\ast }$, attached to the $%
N$-linear connection $D\Gamma \left( N\right) $, have the expressions%
\begin{equation}
\begin{array}{c}
\Delta _{\left( i\right) b}^{\left( a\right) }=-\underset{1}{\overset{}{N}}%
\underset{}{\overset{}{_{\left( i\right) b}^{(a)}}}%
-A_{(i)(c)b}^{(a)(k)}p_{k}^{c},\quad \Delta _{\left( i\right) j}^{\left(
a\right) }=-\underset{2}{\overset{}{N}}\underset{}{\overset{_{{}}}{_{\left(
i\right) j}^{(a)}}}-H_{(i)(c)j}^{(a)(k)}p_{k}^{c},\medskip \\ 
\vartheta _{\left( i\right) \left( b\right) }^{\left( a\right) \left(
j\right) }=\delta _{b}^{a}\delta _{i}^{j}-C_{(i)(c)(b)}^{(a)(k)(j)}p_{k}^{c}.%
\end{array}
\label{Expr-Deflection-Tensors}
\end{equation}
\end{proposition}

\begin{remark}
The polymomentum deflection d-tensor fields (\ref{Expr-Deflection-Tensors})
will be used in a future paper for the construction of a \textbf{generalized
polymomentum electromagnetic geometrical theory} (governed by some \textbf{%
generalized Maxwell equations}), which is derived starting from a given 
\textbf{polymomentum Hamiltonian function of non-autonomous electrodynamic
kind}.
\end{remark}

\section{The torsion of an $N$-linear connection}

\hspace{4mm} Let $D$ be an $N$-linear connection on $E^{\ast }.$ The torsion 
$\mathbb{T}$ of $D$ is given by 
\begin{equation}
\mathbb{T}\left( X,Y\right) =D_{X}Y-D_{Y}X-\left[ X,Y\right] ,\text{ }%
\forall \text{ }X,Y\ \in \chi \left( E^{\ast }\right) .
\label{Torsion-Global}
\end{equation}

It is obvious that the torsion $\mathbb{T}$ can be evaluated by the pairs of
d-vector fields $\left( X^{\mathcal{W}_{\beta }},Y^{\mathcal{W}_{\gamma
}}\right) ,$ $\left( X^{\mathcal{W}_{\beta }},Y^{\mathcal{W}}\right) ,$ $%
\left( X^{\mathcal{W}},Y^{\mathcal{W}}\right) ,$ where $\beta ,\gamma =1,2$
and $\beta \leq \gamma ,$ $\mathcal{W}_{1}=\mathcal{H}_{\mathcal{T}},$ $%
\mathcal{W}_{2}=\mathcal{H}_{M},$ and then we obtain the vector fields: 
\begin{equation*}
\mathbb{T}\left( X^{\mathcal{W}_{\beta }},Y^{\mathcal{W}_{\gamma }}\right)
,\quad\mathbb{T}\left( X^{\mathcal{W}_{\beta }},Y^{\mathcal{W}}\right) ,\quad%
\mathbb{T}\left( X^{\mathcal{W}},Y^{\mathcal{W}}\right) .
\end{equation*}

Because the $N$-linear connection $D$ preserves by parallelism the
distributions $\mathcal{H}_{\mathcal{T}},$ $\mathcal{H}_{M}$ and $\mathcal{W}
$, and the vertical distribution $\mathcal{W}$ is integrable, we find

\begin{proposition}
The following properties of the torsion $\mathbb{T}$ hold good:%
\begin{equation*}
\begin{array}{lll}
h_{M}\mathbb{T}\left( X^{\mathcal{H}_{\mathcal{T}}},Y^{\mathcal{H}_{\mathcal{%
T}}}\right) =0, & h_{M}\mathbb{T}\left( X^{\mathcal{H}_{\mathcal{T}}},Y^{%
\mathcal{W}}\right) =0, & h_{\mathcal{T}}\mathbb{T}\left( X^{\mathcal{H}%
_{M}},Y^{\mathcal{H}_{M}}\right) =0,\medskip \\ 
h_{\mathcal{T}}\mathbb{T}\left( X^{\mathcal{H}_{M}},Y^{\mathcal{W}}\right)
=0, & h_{\mathcal{T}}\mathbb{T}\left( X^{\mathcal{W}},Y^{\mathcal{W}}\right)
=0, & h_{M}\mathbb{T}\left( X^{\mathcal{W}},Y^{\mathcal{W}}\right) =0.%
\end{array}%
\end{equation*}
\end{proposition}

From the preceding statement, we deduce

\begin{proposition}
The torsion tensor field $\mathbb{T}$ of an $N$-linear connection $D$ is
uniquely determined by the following components:%
\begin{equation}
\left\{ 
\begin{array}{l}
\mathbb{T}\left( X^{\mathcal{H}_{\mathcal{T}}},Y^{\mathcal{H}_{\mathcal{T}%
}}\right) =h_{\mathcal{T}}\mathbb{T}\left( X^{\mathcal{H}_{\mathcal{T}}},Y^{%
\mathcal{H}_{\mathcal{T}}}\right) +w\mathbb{T}\left( X^{\mathcal{H}_{%
\mathcal{T}}},Y^{\mathcal{H}_{\mathcal{T}}}\right) ,\medskip \\ 
\mathbb{T}\left( X^{\mathcal{H}_{\mathcal{T}}},Y^{\mathcal{H}_{M}}\right)
=h_{\mathcal{T}}\mathbb{T}\left( X^{\mathcal{H}_{\mathcal{T}}},Y^{\mathcal{H}%
_{M}}\right) +h_{M}\mathbb{T}\left( X^{\mathcal{H}_{\mathcal{T}}},Y^{%
\mathcal{H}_{M}}\right) +\medskip \\ 
\text{ \ \ \ \ \ \ \ \ \ \ \ \ \ \ \ \ \ \ \ \ \ }+w\mathbb{T}\left( X^{%
\mathcal{H}_{\mathcal{T}}},Y^{\mathcal{H}_{M}}\right) ,\medskip \\ 
\mathbb{T}\left( X^{\mathcal{H}_{\mathcal{T}}},Y^{\mathcal{W}}\right) =h_{%
\mathcal{T}}\mathbb{T}\left( X^{\mathcal{H}_{\mathcal{T}}},Y^{\mathcal{W}%
}\right) +w\mathbb{T}\left( X^{\mathcal{H}_{\mathcal{T}}},Y^{\mathcal{W}%
}\right) ,%
\end{array}%
\right.  \label{Torsion-hT+rest}
\end{equation}%
\medskip%
\begin{equation}
\left\{ 
\begin{array}{l}
\mathbb{T}\left( X^{\mathcal{H}_{M}},Y^{\mathcal{H}_{M}}\right) =h_{M}%
\mathbb{T}\left( X^{\mathcal{H}_{M}},Y^{\mathcal{H}_{M}}\right) \ +w\mathbb{T%
}\left( X^{\mathcal{H}_{M}},Y^{\mathcal{H}_{M}}\right) ,\medskip \\ 
\mathbb{T}\left( X^{\mathcal{H}_{M}},Y^{\mathcal{W}}\right) =h_{M}\mathbb{T}%
\left( X^{\mathcal{H}_{M}},Y^{\mathcal{W}}\right) \ +w\mathbb{T}\left( X^{%
\mathcal{H}_{M}},Y^{\mathcal{W}}\right) ,%
\end{array}%
\right.  \label{Torsion-hM+rest}
\end{equation}%
\medskip%
\begin{equation}
\mathbb{T}\left( X^{\mathcal{W}},Y^{\mathcal{W}}\right) =w\mathbb{T}\left(
X^{\mathcal{W}},Y^{\mathcal{W}}\right) ,  \label{Torsion-WW}
\end{equation}%
where in the right part of each equality we have $d$-tensor fields on $%
E^{\ast }.$
\end{proposition}

\begin{definition}
The terms from (\ref{Torsion-hT+rest}), (\ref{Torsion-hM+rest}) and (\ref%
{Torsion-WW}) are called the \textbf{d-tensors of torsion} of the $N$-linear
connection $D$. More exactly, $h_{\mathcal{T}}\mathbb{T}\left( X^{\mathcal{H}%
_{\mathcal{T}}},Y^{\mathcal{H}_{\mathcal{T}}}\right) $ is called the $h_{%
\mathcal{T}}\left( h_{\mathcal{T}}h_{\mathcal{T}}\right) $\textbf{-tensor of
torsion} of $D,$ $h_{M}\mathbb{T}(X^{\mathcal{H}_{\mathcal{T}}},Y^{\mathcal{H%
}_{M}})$ is called the $h_{M}\left( h_{\mathcal{T}}h_{M}\right) $\textbf{%
-tensor of torsion} of $D$ and so on.
\end{definition}

Now, let us suppose that the $N$-linear connection $D$ is given in the
adapted basis (\ref{ad-basis-nlc}) by the coefficients $D\Gamma \left(
N\right) $ from (\ref{Local-descr-D-Nine-Coeff}). In such a context, we have

\begin{theorem}
The torsion $d$-tensors of the $N$-linear connection $D$ on $E^{\ast }$ have
the expressions:%
\begin{equation*}
h_{\mathcal{T}}\mathbb{T}\left( \dfrac{\delta }{\delta t^{b}},\dfrac{\delta 
}{\delta t^{a}}\right) =T_{ab}^{c}\dfrac{\delta }{\delta t^{c}},\quad h_{M}%
\mathbb{T}\left( \dfrac{\delta }{\delta t^{b}},\dfrac{\delta }{\delta t^{a}}%
\right) =T_{ab}^{k}\dfrac{\delta }{\delta x^{k}},
\end{equation*}%
\begin{equation*}
w\mathbb{T}\left( \dfrac{\delta }{\delta t^{b}},\dfrac{\delta }{\delta t^{a}}%
\right) =T_{\left( r\right) ab}^{\left( f\right) }\dfrac{\partial }{\partial
p_{r}^{f}},
\end{equation*}%
\begin{equation*}
h_{\mathcal{T}}\mathbb{T}\left( \dfrac{\delta }{\delta x^{j}},\dfrac{\delta 
}{\delta t^{a}}\right) =T_{aj}^{c}\dfrac{\delta }{\delta t^{c}},\quad h_{M}%
\mathbb{T}\left( \dfrac{\delta }{\delta x^{j}},\dfrac{\delta }{\delta t^{a}}%
\right) =T_{aj}^{k}\dfrac{\delta }{\delta x^{k}},
\end{equation*}%
\begin{equation*}
w\mathbb{T}\left( \dfrac{\delta }{\delta x^{j}},\dfrac{\delta }{\delta t^{a}}%
\right) =T_{\left( r\right) aj}^{\left( f\right) }\dfrac{\partial }{\partial
p_{r}^{f}},
\end{equation*}%
\begin{equation*}
h_{\mathcal{T}}\mathbb{T}\left( \dfrac{\partial }{\partial p_{j}^{b}},\dfrac{%
\delta }{\delta t^{a}}\right) =P_{a\left( b\right) }^{c\left( j\right) }%
\dfrac{\delta }{\delta t^{c}},\quad h_{M}\mathbb{T}\left( \dfrac{\partial }{%
\partial p_{j}^{b}},\dfrac{\delta }{\delta t^{a}}\right) =P_{a\left(
b\right) }^{k\left( j\right) }\dfrac{\delta }{\delta x^{k}},
\end{equation*}%
\begin{equation*}
w\mathbb{T}\left( \dfrac{\partial }{\partial p_{j}^{b}},\dfrac{\delta }{%
\delta t^{a}}\right) =P_{\left( r\right) a\left( b\right) }^{\left( f\right)
\ \left( j\right) }\frac{\partial }{\partial p_{r}^{f}},
\end{equation*}%
\begin{equation*}
h_{\mathcal{T}}\mathbb{T}\left( \dfrac{\delta }{\delta x^{j}},\dfrac{\delta 
}{\delta x^{i}}\right) =T_{ij}^{c}\dfrac{\delta }{\delta t^{c}},\quad h_{M}%
\mathbb{T}\left( \dfrac{\delta }{\delta x^{j}},\dfrac{\delta }{\delta x^{i}}%
\right) =T_{ij}^{k}\dfrac{\delta }{\delta x^{k}},
\end{equation*}%
\begin{equation*}
w\mathbb{T}\left( \dfrac{\delta }{\delta x^{j}},\dfrac{\delta }{\delta x^{i}}%
\right) =T_{\left( r\right) ij}^{\left( f\right) }\dfrac{\partial }{\partial
p_{r}^{f}},
\end{equation*}%
\begin{equation*}
h_{\mathcal{T}}\mathbb{T}\left( \dfrac{\partial }{\partial p_{j}^{b}},\dfrac{%
\delta }{\delta x^{i}}\right) =P_{i\left( b\right) }^{c\left( j\right) }%
\dfrac{\delta }{\delta t^{c}},\quad h_{M}\mathbb{T}\left( \dfrac{\partial }{%
\partial p_{j}^{b}},\dfrac{\delta }{\delta x^{i}}\right) =P_{i\left(
b\right) }^{k\left( j\right) }\dfrac{\delta }{\delta x^{k}},
\end{equation*}%
\begin{equation*}
w\mathbb{T}\left( \dfrac{\partial }{\partial p_{j}^{b}},\dfrac{\delta }{%
\delta x^{i}}\right) =P_{\left( r\right) i\left( b\right) }^{\left( f\right)
\ \left( j\right) }\dfrac{\partial }{\partial p_{r}^{f}},
\end{equation*}%
\begin{equation*}
h_{\mathcal{T}}\mathbb{T}\left( \dfrac{\partial }{\partial p_{j}^{b}},\dfrac{%
\partial }{\partial p_{i}^{a}}\right) =S_{\ \left( a\right) \left( b\right)
}^{c\left( i\right) \left( j\right) }\dfrac{\delta }{\delta t^{c}},\quad
h_{M}\mathbb{T}\left( \dfrac{\partial }{\partial p_{j}^{b}},\dfrac{\partial 
}{\partial p_{i}^{a}}\right) =S_{\ \left( a\right) \left( b\right)
}^{k\left( i\right) \left( j\right) }\dfrac{\delta }{\delta x^{k}},
\end{equation*}%
\begin{equation*}
w\mathbb{T}\left( \dfrac{\partial }{\partial p_{j}^{b}},\dfrac{\partial }{%
\partial p_{i}^{a}}\right) =S_{\left( r\right) \left( a\right) \left(
b\right) }^{\left( f\right) \left( i\right) \left( j\right) }\dfrac{\partial 
}{\partial p_{r}^{f}},
\end{equation*}%
where 
\begin{equation}
\left\{ 
\begin{array}{l}
T_{ab}^{c}=A_{ab}^{c}-A_{ba}^{c},\quad T_{ab}^{k}=0,\quad T_{\left( r\right)
ab}^{\left( f\right) }=R_{\left( r\right) ab}^{\left( f\right) },\medskip \\ 
T_{aj}^{c}=H_{aj}^{c},\quad T_{aj}^{k}=-A_{ja}^{k},\quad T_{\left( r\right)
aj}^{\left( f\right) }=R_{\left( r\right) aj}^{\left( f\right) },\medskip \\ 
P_{a\left( b\right) }^{c\left( j\right) }=C_{a\left( b\right) }^{c\left(
j\right) },\quad P_{a\left( b\right) }^{k\left( j\right) }=0,\quad P_{\left(
r\right) a\left( b\right) }^{\left( f\right) \ \left( j\right) }=B_{\left(
r\right) a\left( b\right) }^{\left( f\right) \ \left( j\right) }+A_{\left(
r\right) \left( b\right) a}^{\left( f\right) \left( j\right) },%
\end{array}%
\right.  \label{Tors-local-restu+hT}
\end{equation}%
\medskip%
\begin{equation}
\left\{ 
\begin{array}{l}
T_{ij}^{c}=0,\quad T_{ij}^{k}=H_{ij}^{k}-H_{ji}^{k},\quad T_{\left( r\right)
ij}^{\left( f\right) }=R_{\left( r\right) ij}^{\left( f\right) },\medskip \\ 
T_{i\left( b\right) }^{c\left( j\right) }=0,\quad P_{i\left( b\right)
}^{k\left( j\right) }=C_{i\left( b\right) }^{k\left( j\right) },\quad
P_{\left( r\right) i\left( b\right) }^{\left( f\right) \ \left( j\right)
}=B_{\left( r\right) i\left( b\right) }^{\left( f\right) \ \left( j\right)
}+H_{\left( r\right) \left( b\right) i}^{\left( f\right) \left( j\right) },%
\end{array}%
\right.  \label{Tors-local-restu+hM}
\end{equation}%
\medskip%
\begin{equation}
S_{\ \left( a\right) \left( b\right) }^{c\left( i\right) \left( j\right)
}=0,\quad S_{\ \left( a\right) \left( b\right) }^{k\left( i\right) \left(
j\right) }=0,\quad S_{\left( r\right) \left( a\right) \left( b\right)
}^{\left( f\right) \left( i\right) \left( j\right) }=-\left( C_{\left(
r\right) \left( a\right) \left( b\right) }^{\left( f\right) \left( i\right)
\left( j\right) }-C_{\left( r\right) \left( b\right) \left( a\right)
}^{\left( f\right) \left( j\right) \left( i\right) }\right)
\label{Tors-local-WW}
\end{equation}%
and the d-tensors $R_{\left( r\right) ab}^{\left( f\right) },$ $R_{\left(
r\right) aj}^{\left( f\right) },$ $R_{\left( r\right) ij}^{\left( f\right)
}, $ $B_{\left( r\right) a\left( b\right) }^{\left( f\right) \ \left(
j\right) } $ and $B_{\left( r\right) i\left( b\right) }^{\left( f\right) \
\left( j\right) }$ are given by (\ref{Formulas-Poisson-brackets}).
\end{theorem}

\begin{proof}
Taking into account the Poisson brackets formulas (\ref{Poisson brackets})
and (\ref{Formulas-Poisson-brackets}), together with the description in the
adapted basis (\ref{ad-basis-nlc}) of the $N$-linear connection $D\Gamma
\left( N\right) $ given by the local coefficients (\ref%
{Local-descr-D-Nine-Coeff}), we successively obtain%
\begin{eqnarray*}
h_{\mathcal{T}}\mathbb{T}\left( \frac{\delta }{\delta t^{b}},\frac{\delta }{%
\delta t^{a}}\right)  &=&h_{\mathcal{T}}D_{\dfrac{\delta }{\delta t^{b}}}%
\frac{\delta }{\delta t^{a}}-h_{\mathcal{T}}D_{\dfrac{\delta }{\delta t^{a}}}%
\frac{\delta }{\delta t^{b}}-h_{\mathcal{T}}\left[ \frac{\delta }{\delta
t^{b}},\frac{\delta }{\delta t^{a}}\right] = \\
&=&\left( A_{ab}^{c}-A_{ba}^{c}\right) \frac{\delta }{\delta t^{c}}.
\end{eqnarray*}%
Consequently, the first equality from (\ref{Tors-local-restu+hT}) is true.
In the sequel, we have%
\begin{eqnarray*}
h_{M}\mathbb{T}\left( \frac{\delta }{\delta x^{j}},\frac{\delta }{\delta
t^{a}}\right)  &=&h_{M}D_{\dfrac{\delta }{\delta x^{j}}}\frac{\delta }{%
\delta t^{a}}-h_{M}D_{\dfrac{\delta }{\delta t^{a}}}\frac{\delta }{\delta
x^{j}}-h_{M}\left[ \frac{\delta }{\delta x^{j}},\frac{\delta }{\delta t^{a}}%
\right] = \\
&=&-A_{ja}^{k}\frac{\delta }{\delta x^{k}}
\end{eqnarray*}%
and the fifth equality from (\ref{Tors-local-restu+hT}) is correct. Then,
for example%
\begin{eqnarray*}
w\mathbb{T}\left( \frac{\partial }{\partial p_{j}^{b}},\frac{\delta }{\delta
t^{a}}\right)  &=&wD_{\dfrac{\partial }{\partial p_{j}^{b}}}\frac{\delta }{%
\delta t^{a}}-wD_{\dfrac{\delta }{\delta t^{a}}}\frac{\partial }{\partial
p_{j}^{b}}-w\left[ \frac{\partial }{\partial p_{j}^{b}},\frac{\delta }{%
\delta t^{a}}\right] = \\
&=&\left( A_{\left( r\right) \left( b\right) a}^{\left( f\right) \left(
j\right) }+B_{\left( r\right) a\left( b\right) }^{\left( f\right) \ \left(
j\right) }\right) \frac{\partial }{\partial p_{r}^{f}}
\end{eqnarray*}%
and the ninth equality from (\ref{Tors-local-restu+hT}) is true. In the same
manner, we obtain the other equalities.
\end{proof}

\begin{corollary}
The torsion\ $\mathbb{T}$\ of an arbitrary $N$-linear connection $D$ on $%
E^{\ast }$ is determined by \textbf{12} effective local $d$-tensors,
arranged in the following table:%
\begin{equation*}
\begin{tabular}{|c|c|c|c|}
\hline
& $h_{\mathcal{T}}$ & $h_{M}$ & $w$ \\ \hline
$h_{\mathcal{T}}h_{\mathcal{T}}$ & $T_{ab}^{c}$ & $0$ & $R_{\left( r\right)
ab}^{\left( f\right) }$ \\ \hline
$h_{M}h_{\mathcal{T}}$ & $T_{aj}^{c}$ & $T_{aj}^{k}$ & $R_{\left( r\right)
aj}^{\left( f\right) }$ \\ \hline
$wh_{\mathcal{T}}$ & $P_{a\left( b\right) }^{c\left( j\right) }$ & $0$ & $%
P_{\left( r\right) a\left( b\right) }^{\left( f\right) \ \left( j\right) }$
\\ \hline
$h_{M}h_{M}$ & $0$ & $T_{ij}^{k}$ & $R_{\left( r\right) ij}^{\left( f\right)
}$ \\ \hline
$wh_{M}$ & $0$ & $P_{i\left( b\right) }^{k\left( j\right) }$ & $P_{\left(
r\right) i\left( b\right) }^{\left( f\right) \ \left( j\right) }$ \\ \hline
$ww$ & $0$ & $0$ & $S_{\left( r\right) \left( a\right) \left( b\right)
}^{\left( f\right) \left( i\right) \left( j\right) }$ \\ \hline
\end{tabular}%
\end{equation*}
\end{corollary}

\begin{example}
For the particular case of the canonical $\overset{B}{N}$-linear Berwald
connection $B\Gamma \overset{B}{\left( N\right) }$, where $\overset{B}{N}=%
\overset{0}{N}$ is given by (\ref{can-nlc-asoc-to-metric}), all $d$-tensors
of torsion vanish except%
\begin{equation*}
\begin{array}{ll}
R_{\left( r\right) ab}^{\left( f\right) }=\varkappa _{gab}^{f}p_{r}^{g}, & 
R_{\left( r\right) ij}^{\left( f\right) }=-\mathbf{r}_{rij}^{s}p_{s}^{f},%
\end{array}%
\end{equation*}%
where $\varkappa _{gab}^{f}(t)$ (resp. $\mathbf{r}_{rij}^{s}(x)$) are the
local curvature tensors of the semi-Rie\-ma\-nni\-an metric $h_{ab}(t)$
(resp. $\varphi _{ij}(x)$).
\end{example}

\section{The curvature of an $N$-linear connection}

\hspace{4mm} Let $D$ be an $N$-linear connection on $E^{\ast }.$ The
curvature $\mathbb{R}$ of $D$ is given by 
\begin{equation}
\mathbb{R}\left( X,Y\right) Z=D_{X}D_{Y}Z-D_{Y}D_{X}Z-D_{\left[ X,Y\right]
}Z,\text{ }\forall \text{ }X,Y,Z\in \chi \left( E^{\ast }\right) .
\label{Curv-Global}
\end{equation}

We will express $\mathbb{R}$\ by his adapted components, taking into account
the decomposition (\ref{decomp-vect-fields}) of the vector fields on $%
E^{\ast }.$ In this direction, we firstly prove

\begin{theorem}
The curvature tensor field $\mathbb{R}$\ of the $N$-linear connection $D$ on 
$E^{\ast }$ has the properties:%
\begin{gather}
\begin{array}{ll}
h_{M}\mathbb{R}\left( X,Y\right) Z^{\mathcal{H}_{\mathcal{T}}}=0, & w\mathbb{%
R}\left( X,Y\right) Z^{\mathcal{H}_{\mathcal{T}}}=0,\medskip \\ 
h_{\mathcal{T}}\mathbb{R}\left( X,Y\right) Z^{\mathcal{H}_{M}}=0, & w\mathbb{%
R}\left( X,Y\right) Z^{\mathcal{H}_{M}}=0,\medskip \\ 
h_{\mathcal{T}}\mathbb{R}\left( X,Y\right) Z^{\mathcal{W}}=0, & h_{M}\mathbb{%
R}\left( X,Z\right) Z^{\mathcal{W}}=0,%
\end{array}%
\medskip  \label{Curv-Global-Propr} \\
\mathbb{R}\left( X,Y\right) Z=h_{\mathcal{T}}\mathbb{R}\left( X,Y\right) Z^{%
\mathcal{H}_{\mathcal{T}}}+h_{M}\mathbb{R}\left( X,Y\right) Z^{\mathcal{H}%
_{M}}+w\mathbb{R}\left( X,Y\right) Z^{\mathcal{W}}.  \notag
\end{gather}
\end{theorem}

\begin{proof}
Because $D$ preserves by parallelism the $\mathcal{H}_{\mathcal{T}}$%
-horizontal, $\mathcal{H}_{M}$-horizontal and vertical distributions, via
the formula (\ref{Curv-Global}), the operator $\mathbb{R}\left( X,Y\right) $
carries $h_{\mathcal{T}}$-horizontal (resp. $h_{M}$-horizontal) vector
fields into $h_{\mathcal{T}}$-horizontal (resp. $h_{M}$-horizontal) vector
fields and the vertical vector fields into vertical vector fields. Thus, the
first six equations from (\ref{Curv-Global-Propr}) hold good. The next one
is an easy consequence of the first six.
\end{proof}

By straightforward calculus, we obtain

\begin{theorem}
The curvature tensor $\mathbb{R}$ of the $N$-linear connection $D$ is
completely determined by \textbf{18} local \textbf{d-tensors of curvature:}%
\begin{equation*}
\mathbb{R}\left( \frac{\delta }{\delta t^{c}},\frac{\delta }{\delta t^{b}}%
\right) \frac{\delta }{\delta t^{a}}=R_{abc}^{d}\frac{\delta }{\delta t^{d}}%
,\quad \mathbb{R}\left( \frac{\delta }{\delta t^{c}},\frac{\delta }{\delta
t^{b}}\right) \frac{\delta }{\delta x^{i}}=R_{ibc}^{l}\frac{\delta }{\delta
x^{l}},
\end{equation*}%
\begin{equation*}
\mathbb{R}\left( \frac{\delta }{\delta t^{c}},\frac{\delta }{\delta t^{b}}%
\right) \frac{\partial }{\partial p_{i}^{a}}=-R_{\left( l\right) \left(
a\right) bc}^{\left( d\right) \left( i\right) }\frac{\partial }{\partial
p_{l}^{d}},
\end{equation*}%
\begin{equation*}
\mathbb{R}\left( \frac{\delta }{\delta x^{k}},\frac{\delta }{\delta t^{b}}%
\right) \frac{\delta }{\delta t^{a}}=R_{abk}^{d}\frac{\delta }{\delta t^{d}}%
,\quad \mathbb{R}\left( \frac{\delta }{\delta x^{k}},\frac{\delta }{\delta
t^{b}}\right) \frac{\delta }{\delta x^{i}}=R_{ibk}^{l}\frac{\delta }{\delta
x^{l}},
\end{equation*}%
\begin{equation*}
\mathbb{R}\left( \frac{\delta }{\delta x^{k}},\frac{\delta }{\delta t^{b}}%
\right) \frac{\partial }{\partial p_{i}^{a}}=-R_{\left( l\right) \left(
a\right) bk}^{\left( d\right) \left( i\right) }\frac{\partial }{\partial
p_{l}^{d}},
\end{equation*}%
\begin{equation*}
\mathbb{R}\left( \frac{\partial }{\partial p_{k}^{c}},\frac{\delta }{\delta
t^{b}}\right) \frac{\delta }{\delta t^{a}}=P_{ab\left( c\right) }^{d\ \left(
k\right) }\frac{\delta }{\delta t^{d}},\quad \mathbb{R}\left( \frac{\partial 
}{\partial p_{k}^{c}},\frac{\delta }{\delta t^{b}}\right) \frac{\delta }{%
\delta x^{i}}=P_{ib\left( c\right) }^{l\ \left( k\right) }\frac{\delta }{%
\delta x^{l}},
\end{equation*}%
\begin{equation*}
\mathbb{R}\left( \frac{\partial }{\partial p_{k}^{c}},\frac{\delta }{\delta
t^{b}}\right) \frac{\partial }{\partial p_{i}^{a}}=-P_{\left( l\right)
\left( a\right) b\left( c\right) }^{\left( d\right) \left( i\right) \ \left(
k\right) }\frac{\partial }{\partial p_{l}^{d}},
\end{equation*}%
\begin{equation*}
\mathbb{R}\left( \frac{\delta }{\delta x^{k}},\frac{\delta }{\delta x^{j}}%
\right) \frac{\delta }{\delta t^{a}}=R_{ajk}^{d\ \ }\frac{\delta }{\delta
t^{d}},\quad \mathbb{R}\left( \frac{\delta }{\delta x^{k}},\frac{\delta }{%
\delta x^{j}}\right) \frac{\delta }{\delta x^{i}}=R_{ijk}^{l}\frac{\delta }{%
\delta x^{l}},
\end{equation*}%
\begin{equation*}
\mathbb{R}\left( \frac{\delta }{\delta x^{k}},\frac{\delta }{\delta x^{j}}%
\right) \frac{\partial }{\partial p_{i}^{a}}=-R_{\left( l\right) \left(
a\right) jk}^{\left( d\right) \left( i\right) }\frac{\partial }{\partial
p_{l}^{d}},
\end{equation*}%
\begin{equation*}
\mathbb{R}\left( \frac{\partial }{\partial p_{k}^{c}},\frac{\delta }{\delta
x^{j}}\right) \frac{\delta }{\delta t^{a}}=P_{aj\left( c\right) }^{d\ \left(
k\right) }\frac{\delta }{\delta t^{d}},\quad \mathbb{R}\left( \frac{\partial 
}{\partial p_{k}^{c}},\frac{\delta }{\delta x^{j}}\right) \frac{\delta }{%
\delta x^{i}}=P_{ij\left( c\right) }^{l\ \left( k\right) }\frac{\delta }{%
\delta x^{l}},
\end{equation*}%
\begin{equation*}
\mathbb{R}\left( \frac{\partial }{\partial p_{k}^{c}},\frac{\delta }{\delta
x^{j}}\right) \frac{\partial }{\partial p_{i}^{a}}=-P_{\left( l\right)
\left( a\right) j\left( c\right) }^{\left( d\right) \left( i\right) \ \left(
k\right) }\frac{\partial }{\partial p_{l}^{d}},
\end{equation*}%
\begin{equation*}
\mathbb{R}\left( \frac{\partial }{\partial p_{k}^{c}},\frac{\partial }{%
\partial p_{j}^{b}}\right) \frac{\delta }{\delta t^{a}}=S_{a\left( b\right)
\left( c\right) }^{d\left( j\right) \left( k\right) }\frac{\delta }{\delta
t^{d}},\quad \mathbb{R}\left( \frac{\partial }{\partial p_{k}^{c}},\frac{%
\partial }{\partial p_{j}^{b}}\right) \frac{\delta }{\delta x^{i}}%
=S_{i\left( b\right) \left( c\right) }^{l\left( j\right) \left( k\right) }%
\frac{\delta }{\delta x^{l}},
\end{equation*}%
\begin{equation*}
\mathbb{R}\left( \frac{\partial }{\partial p_{k}^{c}},\frac{\partial }{%
\partial p_{j}^{b}}\right) \frac{\partial }{\partial p_{i}^{a}}=-S_{\left(
l\right) \left( a\right) \left( b\right) \left( c\right) }^{\left( d\right)
\left( i\right) \left( j\right) \left( k\right) }\frac{\partial }{\partial
p_{l}^{d}},
\end{equation*}%
which we can arrange in the following table:%
\begin{equation}
\begin{tabular}{|c|c|c|c|}
\hline
& $h_{\mathcal{T}}$ & $h_{M}$ & $w$ \\ \hline
$h_{\mathcal{T}}h_{\mathcal{T}}$ & $R_{abc}^{d}$ & $R_{ibc}^{l}$ & $%
R_{\left( l\right) \left( a\right) bc}^{\left( d\right) \left( i\right) }$
\\ \hline
$h_{M}h_{\mathcal{T}}$ & $R_{abk}^{d}$ & $R_{ibk}^{l}$ & $R_{\left( l\right)
\left( a\right) bk}^{\left( d\right) \left( i\right) }$ \\ \hline
$wh_{\mathcal{T}}$ & $P_{ab\left( c\right) }^{d\ \left( k\right) }$ & $%
P_{ib\left( c\right) }^{l\ \left( k\right) }$ & $P_{\left( l\right) \left(
a\right) b\left( c\right) }^{\left( d\right) \left( i\right) \ \left(
k\right) }$ \\ \hline
$h_{M}h_{M}$ & $R_{ajk}^{d}$ & $R_{ijk}^{l}$ & $R_{\left( l\right) \left(
a\right) jk}^{\left( d\right) \left( i\right) }$ \\ \hline
$wh_{M}$ & $P_{aj\left( c\right) }^{d\ \left( k\right) }$ & $P_{ij\left(
c\right) }^{l\ \left( k\right) }$ & $P_{\left( l\right) \left( a\right)
j\left( c\right) }^{\left( d\right) \left( i\right) \ \left( k\right) }$ \\ 
\hline
$ww$ & $S_{a\left( b\right) \left( c\right) }^{d\left( j\right) \left(
k\right) }$ & $S_{i\left( b\right) \left( c\right) }^{l\left( j\right)
\left( k\right) }$ & $S_{\left( l\right) \left( a\right) \left( b\right)
\left( c\right) }^{\left( d\right) \left( i\right) \left( j\right) \left(
k\right) }$ \\ \hline
\end{tabular}
\label{Curvature-Table}
\end{equation}
\end{theorem}

\begin{theorem}
The local curvature $d$-tensors (\ref{Curvature-Table}) are given by the
following formulas:\bigskip

$\left\{ 
\begin{array}{llll}
1. & R_{abc}^{d} & = & \dfrac{\delta A_{ab}^{d}}{\delta t^{c}}-\dfrac{\delta
A_{ac}^{d}}{\delta t^{b}}+A_{ab}^{f}A_{fc}^{d}-A_{ac}^{f}A_{fb}^{d}+C_{a%
\left( f\right) }^{d\left( r\right) }R_{\left( r\right) bc}^{\left( f\right)
},\medskip \\ 
2. & R_{abk}^{d} & = & \dfrac{\delta A_{ab}^{d}}{\delta x^{k}}-\dfrac{\delta
H_{ak}^{d}}{\delta t^{b}}+A_{ab}^{f}H_{fk}^{d}-H_{ak}^{f}A_{fb}^{d}+C_{a%
\left( f\right) }^{d\left( r\right) }R_{\left( r\right) bk}^{\left( f\right)
},\medskip \\ 
3. & P_{ab\left( c\right) }^{d\ \left( k\right) } & = & \dfrac{\partial
A_{ab}^{d}}{\partial p_{k}^{c}}-C_{a\left( c\right) /b}^{d\left( k\right)
}+C_{a\left( f\right) }^{d\left( r\right) }P_{\left( r\right) b\left(
c\right) }^{\left( f\right) \ \left( k\right) },\medskip \\ 
4. & R_{ajk}^{d} & = & \dfrac{\delta H_{aj}^{d}}{\delta x^{k}}-\dfrac{\delta
H_{ak}^{d}}{\delta x^{j}}+H_{aj}^{f}H_{fk}^{d}-H_{ak}^{f}H_{fj}^{d}+C_{a%
\left( f\right) }^{d\left( r\right) }R_{\left( r\right) jk}^{\left( f\right)
},\medskip \\ 
5. & P_{aj\left( c\right) }^{d\ \left( k\right) } & = & \dfrac{\partial
H_{aj}^{d}}{\partial p_{k}^{c}}-C_{a\left( c\right) |j}^{d\left( k\right)
}+C_{a\left( f\right) }^{d\left( r\right) }P_{\left( r\right) j\left(
c\right) }^{\left( f\right) \ \left( k\right) },\medskip \\ 
6. & S_{a\left( b\right) \left( c\right) }^{d\left( j\right) \left( k\right)
} & = & \dfrac{\partial C_{a\left( b\right) }^{d\left( j\right) }}{\partial
p_{k}^{c}}-\dfrac{\partial C_{a\left( c\right) }^{d\left( k\right) }}{%
\partial p_{j}^{b}}+C_{a\left( b\right) }^{f\left( j\right) }C_{f\left(
c\right) }^{d\left( k\right) }-C_{a\left( c\right) }^{f\left( k\right)
}C_{f\left( b\right) }^{d\left( j\right) },%
\end{array}%
\right. $

$\left\{ 
\begin{array}{llll}
7. & R_{ibc}^{l} & = & \dfrac{\delta A_{ib}^{l}}{\delta t^{c}}-\dfrac{\delta
A_{ic}^{l}}{\delta t^{b}}+A_{ib}^{r}A_{rc}^{l}-A_{ic}^{r}A_{rb}^{l}+C_{i%
\left( f\right) }^{l\left( r\right) }R_{\left( r\right) bc}^{\left( f\right)
},\medskip \\ 
8. & R_{ibk}^{l} & = & \dfrac{\delta A_{ib}^{l}}{\delta x^{k}}-\dfrac{\delta
H_{ik}^{l}}{\delta t^{b}}+A_{ib}^{r}H_{rk}^{l}-H_{ik}^{r}A_{rb}^{l}+C_{i%
\left( f\right) }^{l\left( r\right) }R_{\left( r\right) bk}^{\left( f\right)
},\medskip \\ 
9. & P_{ib\left( c\right) }^{l\ \left( k\right) } & = & \dfrac{\partial
A_{ib}^{l}}{\partial p_{k}^{c}}-C_{i\left( c\right) /b}^{l\left( k\right)
}+C_{i\left( f\right) }^{l\left( r\right) }P_{\left( r\right) b\left(
c\right) }^{\left( f\right) \ \left( k\right) },\medskip \\ 
10. & R_{ijk}^{l} & = & \dfrac{\delta H_{ij}^{l}}{\delta x^{k}}-\dfrac{%
\delta H_{ik}^{l}}{\delta x^{j}}%
+H_{ij}^{r}H_{rk}^{l}-H_{ik}^{r}H_{rj}^{l}+C_{i\left( f\right) }^{l\left(
r\right) }R_{\left( r\right) jk}^{\left( f\right) },\medskip \\ 
11. & P_{ij\left( c\right) }^{l\ \left( k\right) } & = & \dfrac{\partial
H_{ij}^{l}}{\partial p_{k}^{c}}-C_{i\left( c\right) |j}^{l\left( k\right)
}+C_{i\left( r\right) }^{l\left( f\right) }P_{\left( f\right) j\left(
c\right) }^{\left( r\right) \ \left( k\right) },\medskip \\ 
12. & S_{i\left( b\right) \left( c\right) }^{l\left( j\right) \left(
k\right) } & = & \dfrac{\partial C_{i\left( b\right) }^{l\left( j\right) }}{%
\partial p_{k}^{c}}-\dfrac{\partial C_{i\left( c\right) }^{l\left( k\right) }%
}{\partial p_{j}^{b}}+C_{i\left( b\right) }^{r\left( j\right) }C_{r\left(
c\right) }^{l\left( k\right) }-C_{i\left( c\right) }^{r\left( k\right)
}C_{r\left( b\right) }^{l\left( j\right) },%
\end{array}%
\right. \bigskip$

$\left\{ 
\begin{array}{llll}
13. & R_{\left( l\right) \left( a\right) bc}^{\left( d\right) \left(
i\right) } & = & 
\begin{array}{l}
\dfrac{\delta A_{\left( l\right) \left( a\right) b}^{\left( d\right) \left(
i\right) }}{\delta t^{c}}-\dfrac{\delta A_{\left( l\right) \left( a\right)
c}^{\left( d\right) \left( i\right) }}{\delta t^{b}}+A_{\left( l\right)
\left( f\right) b}^{\left( d\right) \left( r\right) }A_{\left( r\right)
\left( a\right) c}^{\left( f\right) \left( i\right) }-\medskip  \\ 
-A_{\left( l\right) \left( f\right) c}^{\left( d\right) \left( r\right)
}A_{\left( r\right) \left( a\right) b}^{\left( f\right) \left( i\right)
}+C_{\left( l\right) \left( a\right) \left( f\right) }^{\left( d\right)
\left( i\right) \left( r\right) }R_{\left( r\right) bc}^{\left( f\right) },%
\end{array}%
\medskip  \\ 
14. & R_{\left( l\right) \left( a\right) bk}^{\left( d\right) \left(
i\right) } & = & 
\begin{array}{l}
\dfrac{\delta A_{\left( l\right) \left( a\right) b}^{\left( d\right) \left(
i\right) }}{\delta x^{k}}-\dfrac{\delta H_{\left( l\right) \left( a\right)
k}^{\left( d\right) \left( i\right) }}{\delta t^{b}}+A_{\left( l\right)
\left( f\right) b}^{\left( d\right) \left( r\right) }H_{\left( r\right)
\left( a\right) k}^{\left( f\right) \left( i\right) }-\medskip  \\ 
-H_{\left( l\right) \left( f\right) k}^{\left( d\right) \left( r\right)
}A_{\left( r\right) \left( a\right) b}^{\left( f\right) \left( i\right)
}+C_{\left( l\right) \left( a\right) \left( f\right) }^{\left( d\right)
\left( i\right) \left( r\right) }R_{\left( r\right) bk}^{\left( f\right) },%
\end{array}%
\medskip  \\ 
15. & P_{\left( l\right) \left( a\right) b\left( c\right) }^{\left( d\right)
\left( i\right) \ \left( k\right) } & = & \dfrac{\partial A_{\left( l\right)
\left( a\right) b}^{\left( d\right) \left( i\right) }}{\partial p_{k}^{c}}%
-C_{\left( l\right) \left( a\right) \left( c\right) /b}^{\left( d\right)
\left( i\right) \left( k\right) }+C_{\left( l\right) \left( a\right) \left(
f\right) }^{\left( d\right) \left( i\right) \left( r\right) }P_{\left(
r\right) b\left( c\right) }^{\left( f\right) \ \left( k\right) },\medskip 
\\ 
16. & R_{\left( l\right) \left( a\right) jk}^{\left( d\right) \left(
i\right) } & = & 
\begin{array}{l}
\dfrac{\delta H_{\left( l\right) \left( a\right) j}^{\left( d\right) \left(
i\right) }}{\delta x^{k}}-\dfrac{\delta H_{\left( l\right) \left( a\right)
k}^{\left( d\right) \left( i\right) }}{\delta x^{j}}+H_{\left( l\right)
\left( f\right) j}^{\left( d\right) \left( r\right) }H_{\left( r\right)
\left( a\right) k}^{\left( f\right) \left( i\right) }-\medskip  \\ 
-H_{\left( l\right) \left( f\right) k}^{\left( d\right) \left( r\right)
}H_{\left( r\right) \left( a\right) j}^{\left( f\right) \left( i\right)
}+C_{\left( l\right) \left( a\right) \left( f\right) }^{\left( d\right)
\left( i\right) \left( r\right) }R_{\left( r\right) jk}^{\left( f\right) },%
\end{array}%
\medskip  \\ 
17. & P_{\left( l\right) \left( a\right) j\left( c\right) }^{\left( d\right)
\left( i\right) \ \left( k\right) } & = & \dfrac{\partial H_{\left( l\right)
\left( a\right) j}^{\left( d\right) \left( i\right) }}{\partial p_{k}^{c}}%
-C_{\left( l\right) \left( a\right) \left( c\right) |j}^{\left( d\right)
\left( i\right) \left( k\right) }+C_{\left( l\right) \left( a\right) \left(
f\right) }^{\left( d\right) \left( i\right) \left( r\right) }P_{\left(
r\right) j\left( c\right) }^{\left( f\right) \ \left( k\right) },\medskip 
\\ 
18. & S_{\left( l\right) \left( a\right) \left( b\right) \left( c\right)
}^{\left( d\right) \left( i\right) \left( j\right) \left( k\right) } & = & 
\begin{array}{l}
\dfrac{\partial C_{\left( l\right) \left( a\right) \left( b\right) }^{\left(
d\right) \left( i\right) \left( j\right) }}{\partial p_{k}^{c}}-\dfrac{%
\partial C_{\left( l\right) \left( a\right) \left( c\right) }^{\left(
d\right) \left( i\right) \left( k\right) }}{\partial p_{j}^{b}}+C_{\left(
l\right) \left( f\right) \left( b\right) }^{\left( d\right) \left( r\right)
\left( j\right) }C_{\left( r\right) \left( a\right) \left( c\right)
}^{\left( f\right) \left( i\right) \left( k\right) }-\medskip  \\ 
-C_{\left( l\right) \left( f\right) \left( c\right) }^{\left( d\right)
\left( r\right) \left( k\right) }C_{\left( r\right) \left( a\right) \left(
b\right) }^{\left( f\right) \left( i\right) \left( j\right) }.%
\end{array}%
\end{array}%
\right. \ $
\end{theorem}

\begin{proof}
Taking into account the description in the adapted basis (\ref{ad-basis-nlc}%
) of the $N$-linear connection $D\Gamma \left( N\right) $ given by the local
coefficients (\ref{Local-descr-D-Nine-Coeff}), together with the formulas (%
\ref{Poisson brackets}) and (\ref{Formulas-Poisson-brackets}), we obtain,
for example,%
\begin{equation*}
\mathbb{R}\left( \frac{\partial }{\partial p_{k}^{c}},\frac{\delta }{\delta
t^{b}}\right) \frac{\partial }{\partial p_{i}^{a}}=-P_{\left( l\right)
\left( a\right) b\left( c\right) }^{\left( d\right) \left( i\right) \ \left(
k\right) }\frac{\partial }{\partial p_{l}^{d}}=
\end{equation*}%
\begin{equation*}
=D_{\dfrac{\partial }{\partial p_{k}^{c}}}D_{\dfrac{\delta }{\delta t^{b}}}%
\frac{\partial }{\partial p_{i}^{a}}-D_{\dfrac{\delta }{\delta t^{b}}}D_{%
\dfrac{\partial }{\partial p_{k}^{c}}}\frac{\partial }{\partial p_{i}^{a}}%
-D_{\left[ \dfrac{\partial }{\partial p_{k}^{c}},\dfrac{\delta }{\delta t^{b}%
}\right] }\frac{\partial }{\partial p_{i}^{a}}=\medskip
\end{equation*}%
\begin{equation*}
=-D_{\dfrac{\partial }{\partial p_{k}^{c}}}\left( A_{\left( r\right) \left(
a\right) b}^{\left( f\right) \left( i\right) }\frac{\partial }{\partial
p_{r}^{f}}\right) +D_{\dfrac{\delta }{\delta t^{b}}}\left( C_{\left(
r\right) \left( a\right) \left( c\right) }^{\left( f\right) \left( i\right)
\left( k\right) }\frac{\partial }{\partial p_{r}^{f}}\right) +B_{\left(
r\right) b\left( c\right) }^{\left( f\right) \ \left( k\right) }D_{\dfrac{%
\partial }{\partial p_{r}^{f}}}\frac{\partial }{\partial p_{i}^{a}}=
\end{equation*}%
\begin{eqnarray*}
&=&-\frac{\partial A_{\left( l\right) \left( a\right) b}^{\left( d\right)
\left( i\right) }}{\partial p_{k}^{c}}\frac{\partial }{\partial p_{l}^{d}}%
+A_{\left( r\right) \left( a\right) b}^{\left( f\right) \left( i\right)
}C_{\left( l\right) \left( f\right) \left( c\right) }^{\left( d\right)
\left( r\right) \left( k\right) }\frac{\partial }{\partial p_{l}^{d}}+ \\
&&+\frac{\delta C_{\left( l\right) \left( a\right) \left( c\right) }^{\left(
d\right) \left( i\right) \left( k\right) }}{\delta t^{b}}\frac{\partial }{%
\partial p_{l}^{d}}-C_{\left( r\right) \left( a\right) \left( c\right)
}^{\left( f\right) \left( i\right) \left( k\right) }A_{\left( l\right)
\left( f\right) b}^{\left( d\right) \left( r\right) }\frac{\partial }{%
\partial p_{l}^{d}}-B_{\left( r\right) b\left( c\right) }^{\left( f\right) \
\left( k\right) }C_{\left( l\right) \left( a\right) \left( f\right)
}^{\left( d\right) \left( i\right) \left( r\right) }\frac{\partial }{%
\partial p_{l}^{d}}.
\end{eqnarray*}%
Therefore, we have 
\begin{eqnarray*}
P_{\left( l\right) \left( a\right) b\left( c\right) }^{\left( d\right)
\left( i\right) \ \left( k\right) } &=&\frac{\partial A_{\left( l\right)
\left( a\right) b}^{\left( d\right) \left( i\right) }}{\partial p_{k}^{c}}-%
\underline{\underline{A_{\left( r\right) \left( a\right) b}^{\left( f\right)
\left( i\right) }C_{\left( l\right) \left( f\right) \left( c\right)
}^{\left( d\right) \left( r\right) \left( k\right) }}}- \\
&&-\underline{\underline{\frac{\delta C_{\left( l\right) \left( a\right)
\left( c\right) }^{\left( d\right) \left( i\right) \left( k\right) }}{\delta
t^{b}}+C_{\left( r\right) \left( a\right) \left( c\right) }^{\left( f\right)
\left( i\right) \left( k\right) }A_{\left( l\right) \left( f\right)
b}^{\left( d\right) \left( r\right) }}}+B_{\left( r\right) b\left( c\right)
}^{\left( f\right) \ \left( k\right) }C_{\left( l\right) \left( a\right)
\left( f\right) }^{\left( d\right) \left( i\right) \left( r\right) }.
\end{eqnarray*}%
Now, using the formula of $\mathcal{T}$-horizontal covariant derivative, we
get%
\begin{eqnarray*}
C_{\left( l\right) \left( a\right) \left( c\right) /b}^{\left( d\right)
\left( i\right) \left( k\right) } &=&\underline{\underline{\frac{\delta
C_{\left( l\right) \left( a\right) \left( c\right) }^{\left( d\right) \left(
i\right) \left( k\right) }}{\delta t^{b}}-C_{\left( r\right) \left( a\right)
\left( c\right) }^{\left( f\right) \left( i\right) \left( k\right)
}A_{\left( l\right) \left( f\right) b}^{\left( d\right) \left( r\right)
}+C_{\left( l\right) \left( f\right) \left( c\right) }^{\left( d\right)
\left( r\right) \left( k\right) }A_{\left( r\right) \left( a\right)
b}^{\left( f\right) \left( i\right) }}}+ \\
&&+C_{\left( l\right) \left( a\right) \left( f\right) }^{\left( d\right)
\left( i\right) \left( r\right) }A_{\left( r\right) \left( c\right)
b}^{\left( f\right) \left( k\right) },
\end{eqnarray*}%
and, consequently, interchanging the underlined terms, it follows%
\begin{equation*}
P_{\left( l\right) \left( a\right) b\left( c\right) }^{\left( d\right)
\left( i\right) \ \left( k\right) }=\frac{\partial A_{\left( l\right) \left(
a\right) b}^{\left( d\right) \left( i\right) }}{\partial p_{k}^{c}}%
-C_{\left( l\right) \left( a\right) \left( c\right) /b}^{\left( d\right)
\left( i\right) \left( k\right) }+C_{\left( l\right) \left( a\right) \left(
f\right) }^{\left( d\right) \left( i\right) \left( r\right) }P_{\left(
r\right) b\left( c\right) }^{\left( f\right) \ \left( k\right) },
\end{equation*}%
where we also used the last formula from (\ref{Tors-local-restu+hT}).
Obviously, this is the $15^{\text{-}th}$ relation of the above lot.

The other equalities are given in the same manner.
\end{proof}

\begin{example}
For the canonical Berwald $\overset{0}{N}$-linear connection $B\Gamma \left( 
\overset{0}{N}\right) $, which is associated to the pair of semi-Riemannian
metrics $h_{ab}(t)$ and $\varphi _{ij}(x)$, all local curvature $d$-tensors
vanish, except%
\begin{equation*}
R_{abc}^{d}=\varkappa _{abc}^{d},\quad R_{\left( l\right) \left( a\right)
bc}^{\left( d\right) \left( i\right) }=-\delta _{l}^{i}\varkappa
_{abc}^{d},\quad R_{ijk}^{l}=\mathbf{r}_{ijk}^{l},\quad R_{\left( l\right)
\left( a\right) jk}^{\left( d\right) \left( i\right) }=\delta _{a}^{d}%
\mathbf{r}_{ljk}^{i},
\end{equation*}%
where $\varkappa _{abc}^{d}(t)\ $(resp. $\mathbf{r}_{ijk}^{l}(x)$) are the
local curvature tensors of the semi-Rie\-ma\-nni\-an metric $h_{ab}(t)$
(resp. $\varphi _{ij}(x)$).
\end{example}

\textbf{Author's address: }Gheorghe ATANASIU, Str. Gh. Baiulescu, Nr. 11, Bra%
\c{s}ov, BV 500107, Romania.

\textbf{E-mail}: gh\_atanasiu@yahoo.com

\textbf{Place of work: }University "Transilvania" of Bra\c{s}ov, Faculty of
Ma\-the\-ma\-tics and Informatics, Department of Algebra, Geometry and
Differential Equations.

\bigskip

\textbf{Author's address: }Mircea NEAGU, Str. L\u{a}m\^{a}i\c{t}ei, Nr. 66,
Bl. 93, Sc. G, Ap. 10, Bra\c{s}ov, BV 500371, Romania.

\textbf{E-mail}: mirceaneagu73@yahoo.com

\textbf{Place of work: }University "Transilvania" of Bra\c{s}ov, Faculty of
Ma\-the\-ma\-tics and Informatics, Department of Algebra, Geometry and
Differential Equations.

\end{document}